\documentclass[12pt, a4paper]{amsart}
\usepackage{amsmath}
\usepackage{geometry,amsthm,graphics,tabularx,amssymb,shapepar}
\usepackage{amscd}
\usepackage[all]{xypic}

\newcommand{\GL}{{\mathrm{GL}}}

\newcommand{\Hom}{{\mathrm{Hom}}}

\newcommand{\SL}{{\mathrm{SL}}}

\newcommand{\SO}{{\mathrm{SO}}}
\newcommand{\GO}{{\mathrm{GO}}}

\newcommand{\GU}{{\mathrm{GU}}}

\newcommand{\sgn}{\operatorname{sgn}}

\newcommand{\od}{\operatorname{d}}

\newcommand{\oL}{\operatorname{L}}

\newcommand{\oH}{\operatorname{H}}
\newcommand{\oO}{\operatorname{O}}

\newcommand{\oT}{\operatorname{T}}
\newcommand{\oU}{\operatorname{U}}

\newcommand{\oW}{\textit{W}}

\newcommand{\g}{\mathfrak g}

\renewcommand{\k}{\mathfrak k}
\newcommand{\h}{\mathfrak h}

\renewcommand{\b}{\mathfrak b}
\renewcommand{\c}{\mathfrak c}
\newcommand{\n}{\mathfrak n}

\renewcommand{\l}{\mathfrak l}
\renewcommand{\t}{\mathfrak t}
\newcommand{\s}{\mathfrak s}

\renewcommand{\o}{\mathfrak o}

\renewcommand{\sl}{\mathfrak s \mathfrak l}
\newcommand{\gl}{\mathfrak g \mathfrak l}

\newcommand{\Z}{\mathbb{Z}}
\newcommand{\C}{\mathbb{C}}
\newcommand{\R}{\mathbb R}

\newcommand{\K}{\mathbb{K}}

\newcommand{\abs}[1]{\lvert#1\rvert}

\newcommand{\la}{\langle}
\newcommand{\ra}{\rangle}

\newcommand{\be}{\begin {equation}}
\newcommand{\ee}{\end {equation}}
\newcommand{\bee}{\begin {equation*}}
\newcommand{\eee}{\end {equation*}}

\newcommand{\cf}{\emph{cf.}~}

\theoremstyle{Theorem}

\theoremstyle{Theorem}
\newtheorem{introconjecture}{Conjecture}

\newtheorem{introproposition}[introconjecture]{Proposition}
\newtheorem{introtheorem}[introconjecture]{Theorem}

\theoremstyle{Theorem}
\newtheorem{lem}{Lemma}[section]

\newtheorem{prpl}[lem]{Proposition}

\theoremstyle{Theorem}

\theoremstyle{Plain}

\theoremstyle{remark}
\newtheorem*{example}{Example}

\theoremstyle{Definition}

\begin{document}

\title[The nonvanishing hypothesis]{The nonvanishing hypothesis at infinity for Rankin-Selberg convolutions}

\author{Binyong Sun}

\address{Institute of Mathematics and Hua Loo-Keng Key Laboratory of Mathematics, AMSS, Chinese Academy of Sciences, Beijing, 100190, China} \email{sun@math.ac.cn}



\begin{abstract}
We prove the nonvanishing hypothesis at infinity for Rankin-Selberg convolutions for $\GL(n)\times \GL(n-1)$.
\end{abstract}

 \maketitle


\section{Introduction}\label{s1}

The goal of this paper is to solve a long awaited problem which appears in the arithmetic study of special values of L-functions. It is called  the nonvanishing hypothesis in the literature.

We first treat the more involved case of real groups, and leave the complex case to Section \ref{complexg}. Fix an integer $n\geq 2$, and fix a decreasing sequence
\[
  \mu=(\mu_1\geq \mu_2\geq \cdots \geq \mu_n)\in \Z^n.
\]
Denote by $F_{\mu}$ the irreducible algebraic representation of $\GL_n(\C)$ of highest weight $\mu$.
Denote by $\Omega(\mu)$ the set of isomorphism
classes of irreducible Casselman-Wallach representations $\pi$ of
$\GL_{n}(\R)$ such that
\begin{itemize}
                           \item
                            $\pi|_{{\operatorname{SL}_n^{\pm}}(\R)}$ is unitarizable and
                            tempered; and
                           \item the total relative Lie algebra cohomology
                           \begin{equation}\label{cohom}
                             \oH^*(\gl_{n}(\C),\GO(n)^\circ;
                           F_\mu^\vee\otimes \pi)\neq 0,
                           \end{equation}
                             \end{itemize}
where
\[
  {\operatorname{SL}_n^{\pm}}(\R):=\{g\in \GL_n(\R)\mid \det(g)=\pm 1\},
  \]
  and
\[
  \GO(n):=\{g\in \GL_n(\R)\mid g^{\mathrm t} g\textrm{ is a scalar matrix}\}.
  \]
Here and henceforth, a superscript ``$\,\mathrm t\,$" over a matrix indicates its transpose, a superscript ``$\,\circ\,$" over a Lie group indicates its identity connected component, and ``$\,^\vee$" indicates the contragredient representation. Recall that a representation of a real reductive group is called a Casselman-Wallach representation if it is smooth, Fr\'{e}chet, of moderate growth, and its Harish-Chandra module has finite length. The reader may consult \cite{Cass}, \cite[Chapter
11]{Wa2} or \cite{BK} for details about Casselman-Wallach representations. As is quite common, we do not distinguish a representation with its
underlying space, or an irreducible representation with its isomorphism class.

The set $\Omega(\mu)$ is explicitly determined by Vogan-Zuckerman theory \cite{VZ} of cohomological representations. In particular \cite[Section 3]{Clo},
\begin{equation}\label{pimu3}
  \#(\Omega(\mu))=\left\{
                \begin{array}{ll}
                  0, & \hbox{if $\mu$ is not pure;} \\
                  1, & \hbox{if $\mu$ is pure and $n$ is even;} \\
                  2, & \hbox{if $\mu$ is pure and $n$ is odd.}
                \end{array}
              \right.
\end{equation}
Here ``$\mu$ is pure" means that
\begin{equation}\label{puremu}
  \mu_{1}+\mu_{n}=\mu_{2}+\mu_{n-1}=\cdots=\mu_{n}+\mu_{1}.
\end{equation}
Recall the sign character
\[
\sgn:=\det\,\abs{\det}^{-1}
\]
of a real general linear group.   In the second case of \eqref{pimu3}, the only representation in $\Omega(\mu)$ is isomorphic to its twist by the sign character. In the third case of \eqref{pimu3}, the two representations in $\Omega(\mu)$ are twists of each other by the sign character.

 Assume that $\mu$ is pure, and let $\pi_\mu\in \Omega(\mu)$. Put
\[
  b_n:={\lfloor}\frac{n^2}{4}{\rfloor}.
\]
Then \cite[Lemma 3.14]{Clo}
\[
  \oH^b(\gl_{n}(\C),\GO(n)^\circ;
                           F_\mu^\vee\otimes \pi_\mu)=0,\qquad \textrm{if $b<b_n$,}
\]
and
\[
  \dim \oH^{b_n}(\gl_{n}(\C),\GO(n)^\circ;
                           F_\mu^\vee\otimes \pi_\mu)=\left\{
                                                   \begin{array}{ll}
                                                     2, & \hbox{if $n$ is even;} \\
                                                     1, & \hbox{if $n$ is odd.}
                                                   \end{array}
                                                 \right.
\]
Write $\g\o_n$ for the complexified Lie algebra of $\GO(n)$. The group $\GO(n)$ acts linearly on both $\gl_n(\C)/\g\o_n$ and $F_\mu^\vee\otimes \pi_\mu$. Passing to cohomology, we get a representation of $\GO(n)/\GO(n)^\circ$ on
\begin{equation}\label{acth}
  \oH(\pi_\mu):=\oH^{b_n}(\gl_{n}(\C),\GO(n)^\circ;
                           F_\mu^\vee\otimes \pi_\mu).
\end{equation}
The sign characters induce  characters of some subquotients of real general linear groups (such as $\GO(n)/\GO(n)^\circ$). We still use $\sgn$ to denote these induced characters. Then as a representation of $\GO(n)/\GO(n)^\circ$ \cite[Equation (3.2)]{Mah},
\begin{equation}\label{ohpi}
  \oH(\pi_\mu)\cong \left\{
                      \begin{array}{ll}
                        \sgn^0\oplus \sgn, & \hbox{if $n$ is even;} \\
                        \sgn^{\pi_\mu(-1)+\mu_1+\mu_2+\cdots +\mu_n}, & \hbox{if $n$ is odd,}
                      \end{array}
                    \right.
\end{equation}
where $\pi_\mu(-1)$ denotes the scalar by which $-1\in \GL_n(\R)$ acts on the representation $\pi_\mu$.

We also fix a decreasing sequence
\be\label{intrnu}
  \nu=(\nu_1\geq \nu_2\geq \cdots \geq \nu_{n-1})\in \Z^{n-1}.
\ee
Define $F_\nu$ and $\Omega(\nu)$ similarly. Assume that $\nu$ is pure, and let $\pi_\nu\in \Omega(\nu)$. Similar to $\oH(\pi_\mu)$, we have a cohomology space $\oH(\pi_\nu)$, which is a representation of $\GO(n-1)/\GO(n-1)^\circ$ of dimension $1$ or $2$.

Recall that an element of  $\frac{1}{2}+\Z$ is called a critical place for $\pi_\mu\times \pi_\nu$ if it is not a pole of the local L-function $\oL(s, \pi_\mu \times \pi_\nu)$ or $\oL(1-s, \pi_\mu^\vee \times \pi_\nu^\vee)$. Assume that $\mu$ and $\nu$ are compatible in the sense that there is an integer $j$ such that
\begin{equation}\label{homf}
  \Hom_{H_\C}(F_\xi^\vee, {\det}^j) \neq 0,
\end{equation}
where
\begin{equation}\label{f}
 F_\xi:=F_\mu \otimes F_\nu,
\end{equation}
 and
\[
  H_\C:=\GL_{n-1}(\C)
\]
is viewed as a subgroup of
\[
  G_\C:=\GL_{n}(\C)\times\GL_{n-1}(\C)
\]
 via the embedding
\begin{equation}\label{emh}
  g\mapsto \left(\left[
                      \begin{array}{cc}
                        g & 0 \\
                        0 & 1 \\
                      \end{array}
                    \right], g \right).
\end{equation}
It is proved in \cite[Theorem 2.3]{KS} that  $\frac{1}{2}+j$ is a critical place for $\pi_\mu\times \pi_\nu$, and conversely, all critical places are of this form (under the assumption that $\mu$ and $\nu$ are compatible). Fix a nonzero element $\phi_F$ of the hom space \eqref{homf}, which is unique up to scalar multiplication.

Write
\begin{equation}\label{ktilde}
   G:=\GL_n(\R)\times \GL_{n-1}(\R)\quad\textrm{and}\quad  \tilde K:=\GO(n)\times \GO(n-1)\subset G,
\end{equation}
and write their respective subgroups
\[
  H:=G\cap H_\C=\GL_{n-1}(\R) \quad \textrm{and}\quad C:=H\cap \tilde K=\oO(n-1).
\]

Since  $\frac{1}{2}+j$ is a critical place for $\pi_\mu\times \pi_\nu$, the Rankin-Selberg integrals (see \cite[Section 2]{Jac}) for $\pi_\mu\times \pi_\nu$ are holomorphic at $\frac{1}{2}+j$, and produce a nonzero element
\begin{equation}\label{homi}
  \phi_\pi\in \Hom_H(\pi_\xi, \abs{\det}^{-j}),
\end{equation}
where
\[
  \pi_\xi:=\pi_\mu\widehat \otimes \pi_\nu \quad \qquad (\textrm{the completed projective tensor product})
\]
is a Casselman-Wallach representation of $G$.

As usual, we use the corresponding lower case gothic letter to indicate the complexified Lie algebra of a Lie group. We formulate the nonvanishing hypothesis at the real place as follows.
\begin{introtheorem}\label{main}
By restriction of cohomology, the $H$-equivariant linear functional
\[
 \phi_F\otimes \phi_\pi:  F_\xi^\vee\otimes \pi_\xi \rightarrow \sgn^j={\det}^{j}\otimes \abs{\det}^{-j}
\]
induces a linear map
\begin{equation}\label{homh}
   \oH^{b_n+b_{n-1}}(\g, \tilde K^\circ; F_\xi^\vee\otimes \pi_\xi)\rightarrow
   \oH^{b_n+b_{n-1}}(\h, C^\circ; {\sgn}^j)
\end{equation}
which is nonzero.

\end{introtheorem}

Note that
\[
\dim (\h/\c)=\frac{n(n-1)}{2}=b_n+b_{n-1}.
\]
Therefore, Poincar\'{e} duality (see \cite[Proposition 7.6]{BN}) implies that the space of the right hand side of \eqref{homh} is one dimensional. As in \eqref{acth}, it carries a representation of $C/C^\circ$. This representation is isomorphic to $\sgn^{n+j}$.

By K\"{u}nneth formula, the left hand side of \eqref{homh} is canonically isomorphic to
\[
  \oH(\pi_\mu)\otimes \oH(\pi_\nu).
\]
By \eqref{ohpi} and its analog for $\oH(\pi_\nu)$, the above space is isomorphic to
\[
 \sgn^0\oplus \sgn=\sgn^{n+j}\oplus \sgn^{n+j+1},
\]
as a representation of
\[
  C/C^\circ\subset \tilde K/\tilde K^\circ=(\GO(n)/\GO(n)^\circ)\times (\GO(n-1)/\GO(n-1)^\circ).
\]
Since the linear map \eqref{homh} is $C/C^\circ$-equivariant, it has to vanish on

\[
  \sgn^{n+j+1}\subset  \oH(\pi_\mu)\otimes \oH(\pi_\nu).
\]
Thus, Theorem \ref{main} amounts to saying that the linear functional \eqref{homh} does not vanish on the one dimensional space
\[
  \sgn^{n+j}\subset  \oH(\pi_\mu)\otimes \oH(\pi_\nu).
\]

The nonvanishing hypothesis is vital to the arithmetic study  of critical values of higher degree L-functions and to the constructions of higher degree p-adic L-functions, via the Rankin-Selberg method and modular symbols. As is emphasized by Raghuram-Shahidi \cite{RS1}, ``\emph{it is an important technical problem to be able to prove
this nonvanishing hypothesis}". When $n=2$, Theorem \ref{main} is due to Hecke \cite{He1, He2, He3}. The higher rank case of Theorem \ref{main} has been expected by Kazhdan and Mazur since 1970's. For $n=3$, Theorem \ref{main} is proved by Mazur \cite[Theorem 3.8]{Sch} when both $F_\mu$ and $F_\nu$ are trivial representations, and by Kasten-Schmidt \cite[Theorem B]{KS} in general. In the literature, many theorems on special values of L-functions have been proved
under the assumption that Theorem \ref{main} is valid. See Kazhdan-Mazur-Schmidt \cite{KMS}, Mahnkopf
\cite{Mah}, Raghuram \cite{Rag}, Kasten-Schmidt \cite{KS}, Januszewski \cite{Jan1}, Raghuram-Shahidi \cite{RS1, RS}, and
Schmidt \cite{Sch, Sch2}. For more recent works using Theorem \ref{main} (and Theorem \ref{main2} of Section \ref{complexg}), see  Grobner-Harris \cite{GH}, Raghuram \cite{Rag2}, and Januszewski \cite{Jan2, Jan3, Jan4}.

Let us briefly explain the idea of the proof of Theorem \ref{main}.
Write $K:=\oO(n)\times \oO(n-1)$. Recall from \cite[Theorem 4.9]{V2} that every irreducible Casselman-Wallach representation of $G$ has a unique minimal $K$-type, and the minimal $K$-type occurs with multiplicity one in the representation. Denote by $\sigma_\xi$ the unique minimal $K$-type of $\pi_\xi$, and view it as a subspace of $\pi_\xi$.  Besides some ``minor problems" in classical invariant theory, the main problem of the proof of Theorem \ref{main} is to show that the nonzero functional $\phi_\pi$ in the space
\be\label{homhh}
 \Hom_H(\pi_\xi, \abs{\det}^{-j})
 \ee
  does not vanish on the minimal $K$-type $\sigma_\xi\subset \pi_\xi$. The first key ingredient of the proof is the multiplicity one theorem as proved in \cite[Theorem B]{AzG} and \cite[Theorem B]{SZ}, which implies that the space \eqref{homhh} is one dimensional. Thus it suffices to produce an element of the space \eqref{homhh} which does not vanish on the minimal $K$-type $\sigma_\xi$.

Ignoring  the convergence problem, there is another way to produce elements of \eqref{homhh} besides using the Rankin-Selberg integrals, that is, by using the matrix coefficient integrals
\be\label{maint0}
\begin{array}{rcl}
\la \,\, ,\,\ra_{\mathrm{mc}} \,:\, \pi_\xi \times \pi_\xi &\rightarrow &\C,\\
    (u, v)&\mapsto & \int_H \la h.u, v\ra_\xi\,\abs{\det(h)}^j\,\mathrm{d}h,
  \end{array}
\ee
where $\mathrm d h$ is a Haar measure on $H$, and $\la\,\,,\,\ra_\xi$ denotes an $\SL_n^{\pm}(\R)\times \SL_{n-1}^\pm(\R)$-invariant continuous inner product on $\pi_\xi$.
The second key ingredient of the proof is then the following positivity result, which is proved by the author \cite[Theorem 1.5]{Sun1} in a more general setting.

\begin{introproposition}\label{positiveintr} For every nonzero vector $u$ in the minimal $K$-type $\sigma_\xi$ of $\pi_\xi$, the inequality
\[
  \la g.u, u\ra_\xi>0
\]
holds for all $g\in G=\GL_n(\R)\times \GL_{n-1}(\R)$ which is a pair of positive definite matrices.
\end{introproposition}

\begin{example}
 Let us explicate  Proposition \ref{positiveintr} in the simple case when $n=2$ and $F_\mu$ is the trivial representation. In this case,  $\pi_\mu$ is the relative discrete series of $\GL_2(\R)$ of weight $2$. The  fractional linear transformation yields a representation of $\GL_2(\R)$ on the Hilbert space of all holomorphic 1-forms $\omega$ on $\C\setminus \R$ such that
\[
  \int_{\C\setminus \R} \mathbf i\, \omega\wedge \bar \omega <\infty \qquad (\mathbf i=\sqrt{-1}\in \C).
\]
This is an irreducible unitary representation, and $\pi_\mu$ is realized as its space of smooth vectors. Let $\omega_+$ denote the holomorphic 1-form on $\C\setminus \R$ which vanishes on the lower half plane, and whose restriction to the upper half plane equals
\[
   \od \left(  \frac{z-\mathbf i}{-\mathbf iz+1}\right)\qquad(z\textrm{ denotes the standard coordinate function on $\C$}).
\]
Likewise, let
$\omega_-$ denote the holomorphic 1-form on $\C\setminus \R$ which vanishes on the upper half plane, and whose restriction to the lower half plane equals
\[
   \od \left(  \frac{-z-\mathbf i}{\mathbf iz+1}\right).
\]
Then $\omega_+$ and $\omega_-$ span the minimal $\oO(2)$-type of $\pi_\mu$, and Proposition \ref{positiveintr} is equivalent to saying that
\be\label{gl2p}
   \int_{\C\setminus \R} \mathbf i\, (g. \omega_+)\wedge \overline \omega_+>0\qquad\textrm{and}\qquad \int_{\C\setminus \R} \mathbf i\, (g. \omega_-)\wedge \overline \omega_->0,
\ee
for all $g\in \GL_2(\R)$ which is a positive definite matrix. Although not obvious, the two inequalities of  \eqref{gl2p} may be verified by hand.
The reader is referred to  \cite[Chapter I, \S 6]{K} for more details concerning this example. See \cite{Fl} for more information on matrix coefficients of discrete series representations.
\end{example}

Returning  to the general case, invariant theory shows that there is a nonzero $C$-invariant vector $u_\xi\in \sigma_\xi$, which is  unique up to scalar multiplication. Using the Cartan decomposition, Proposition \ref{positiveintr} implies that $\la u_\xi, u_\xi\ra_{\mathrm{mc}}>0$. Thus, assuming that the map \eqref{maint0} is well-defined and continuous,  we get an element of the space \eqref{homhh}, namely $\la \,\cdot\,, u_\xi\ra_{\mathrm{mc}}$, which does not vanish on the minimal $K$-type $\sigma_\xi$.

In the special case when $\pi_\xi$ is unitarizable and $j=0$, the integrals in \eqref{maint0} do converge, and $\la\,\,,\,\ra_{\mathrm{mc}}$ is a continuous Hermitian form on $\pi_\xi$ (\cf Lemma \ref{ematrix}). Therefore the above argument does work in this special case. In order to overcome the convergence problem in the general case, we introduce an auxiliary unitarizable representation
$\pi_{\tilde \xi}:=\pi_{\tilde \mu}\widehat \otimes \pi_{\tilde \nu}$
of $G$ such that  $\pi_\xi$  is uniquely realized as a subrepresentation of $F_\xi\otimes \pi_{\tilde \xi}$, where
\[
   \tilde \mu:=(\mu_1-\mu_n, \mu_2-\mu_{n-1},\cdots, \mu_n-\mu_1),\qquad  \pi_{\tilde \mu}\in \Omega(\tilde \mu),
\]
and similarly for $\tilde \nu$ and $\pi_{\tilde \nu}$.
The minimal $K$-type $\sigma_\xi$ of $\pi_\xi$ is explicitly described as a subspace of $F_\xi\otimes \sigma_{\tilde \xi}$, where $\sigma_{\tilde \xi}$ denotes the minimal $K$-type of $\pi_{\tilde \xi}$. The above argument shows that there is an element $\phi_\pi'\in \Hom_H(\pi_{\tilde \xi}, \sgn^{-j})$ which does not vanish on $\sigma_{\tilde \xi}$. Take a nonzero element $\phi_F'\in \Hom_{H_\C}(F_\xi, {\det}^{-j})$. Then some invariant theoretic considerations show that $\phi_F'\otimes \phi_\pi'$ restricts to an element of $\Hom_H(\pi_\xi, \abs{\det}^{-j})$ which does not vanish on the minimal $K$-type $\sigma_\xi$.

We now comment on the organization of this paper. Section \ref{s2} is devoted to some preliminary results on classical invariant theory. In section \ref{cohom}, we realize the representation $\pi_\mu$ by cohomological induction, and identify its minimal $\oO(n)$-type inside the tensor product $F_\mu\otimes \pi_{\tilde \mu}$, for some auxiliary representation $\pi_{\tilde \mu}\in \Omega(\tilde \mu)$. The nonvanishing of $\phi_\pi$ on the minimal $K$-type is then proved in Section \ref{ctest}. With these preparations, Theorem \ref{main} is finally proved in Section \ref{pmanin}. In Section \ref{complexg}, the analog of Theorem \ref{main} for complex groups is obtained with a sketched proof.

\vskip 5pt

\section{Finite dimensional representations}\label{s2}

\subsection{Some root systems}\label{str}
For every integer $k\geq 1$, we inductively define a Lie algebra embedding
\begin{equation}\label{vtk}
\gamma_k: \C^k\rightarrow
\g_k:=\gl_k(\C)
\end{equation}
as follows: $\gamma_1$ is the identity map,
$\gamma_2$ is given by
\[
  (a_1,a_2)\mapsto
  \left[
  \begin{array}{cc}
    \frac{a_1+a_{2}}{2} & \frac{a_1-a_{2}}{2\mathbf i}  \\
     \frac{a_{2}-a_1}{2\mathbf i} & \frac{a_{2}+a_1}{2}  \\
                  \end{array}
\right],  \qquad(\mathbf i=\sqrt{-1}\in \C)
\]
and if $k\geq 3$, $\gamma_k$ is given by
\[
   (a_1,a_2,\cdots, a_k)\mapsto
  \left[
  \begin{array}{cc}
    \gamma_{k-2}(a_1,a_2,\cdots, a_{k-2}) & 0  \\
     0& \gamma_{2}(a_{k-1}, a_k)  \\
  \end{array}
\right].
\]

Denote by $\t_k$ the image of $\gamma_k$. It is a fundamental Cartan subalgebra for the group $\GL_k(\R)$ in the following sense:
\[
\left\{
  \begin{array}{l}
    \bar{\t_k}=\t_k;\smallskip \\
    \theta_k (\t_k)=\t_k; \smallskip \\
    \t_k^\mathrm c:=\t_k\cap \o_k\textrm{ is a Cartan subalgebra of the orthogonal Lie algebra $\o_k:=\o_k(\C)$}.
  \end{array}
\right.
\]
Here and henceforth, an overbar indicates the complex conjugation in various contexts; and
\[
\theta_k :\g_k\rightarrow \g_k,\qquad x\mapsto -x^{\mathrm t}
\]
denotes the Caran involution corresponding to $\oO(k)$.

Identify $\t_k$ with $\C^k$ via $\gamma_k$. Then its dual space $\t_k^*$ is also identified with $\C^k$. The root system of $\g_k$ with respect to $\t_k$ is
\begin{equation}\label{phik}
  \Delta_k:=\{\pm (\mathrm e^r_k-{\mathrm e}_k^s)\mid 1\leq r<s\leq k\}
\end{equation}
where $\mathrm e_k^1$, $\mathrm e_k^2$,$\cdots$, $\mathrm e^k_k$  denote the standard basis of $\C^k=\t_k^*$.

For every $\lambda\in \t_k^*$, write $[\lambda]\in {\t_k^{\mathrm c}}^*$ for its restriction to $\t_k^{\mathrm c}$.
Put
\[
  \Lambda_k:=\left \{r\in \{2,3,\cdots,k\}\mid r\equiv k \mod 2\right \}.
\]
Then
\[
\left\{
  \begin{array}{l}
   \,\! [{\mathrm e}^{r-1}_k]+[{\mathrm e}^r_k]=0, \quad \textrm{ for all $r\in \Lambda_k$;}\smallskip \\
   \, \! [{\mathrm e}_k^1]=0, \quad \textrm{if $k$ is odd.}
  \end{array}
\right.
\]
Therefore, the root system of $\g_k$ with respect to $\t_k^{\mathrm c}$ is
\[
 [{\Delta}_k]=\left\{
          \begin{array}{ll}
            \{\pm [{\mathrm e}_k^r]\pm [{\mathrm e}_k^s]\mid r<s\}\cup \{\pm 2 [{\mathrm e}_k^t]\}, & \hbox{if $k$ is even;} \smallskip\\
            \{\pm [{\mathrm e}_k^r]\pm [{\mathrm e}_k^s]\mid r<s\}\cup \{\pm [{\mathrm e}_k^t], \,\pm 2 [{\mathrm e}_k^t]\} , & \hbox{if $k$ is odd,}
          \end{array}
        \right.
\]
where $r,s,t$ run through all elements in $\Lambda_k$.

 Note that the Weyl group
\[
   \frac{\textrm{the normalizer of $\t_k^{\mathrm c}$ in $\oO(k)$}}{\textrm{the centralizer of $\t_k^{\mathrm c}$ in $\oO(k)$}}
\]
acts simply transitively on the set of all positive systems in $[{\Delta}_k]$. Fix such a positive system:
\begin{equation}\label{pso0}
 [{\Delta}_k]^+:=\left\{
          \begin{array}{ll}
            \{\pm [{\mathrm e}_k^r]+[{\mathrm e}_k^s]\mid r<s\}\cup \{2 [{\mathrm e}_k^t]\}, & \hbox{if $k$ is even;} \smallskip\\
            \{\pm [{\mathrm e}_k^r]+[{\mathrm e}_k^s]\mid r<s\}\cup \{[{\mathrm e}_k^t],\,2 [{\mathrm e}_k^t]\}, & \hbox{if $k$ is odd,}
          \end{array}
        \right.
\end{equation}
where  $r,s,t $ run through all elements in $\Lambda_k$. Denote by $\b_k$ the corresponding Borel subalgebra of $\g_k$. It is ``theta stable" in the sense that
\[
   \theta_k(\b_k)=\b_k\quad\textrm{and}\quad \b_k\cap \overline{\b_k}=\t_k.
\]
Put $\b_k^{\mathrm c}:=\b_k\cap \o_k$, which is a Borel subalgebra of $\o_k$. Denote by $\n_k$ and $\n_k^{\mathrm c}$ the nilpotent radicals of $\b_k\cap [\g_k,\g_k]$ and $\b_k^{\mathrm c}\cap [\o_k,\o_k]$, respectively.

Given an element  $\lambda\in \t_k^*$ which is real, namely $\lambda\in \R^k\subset \C^k=\t_k^*$, put
\begin{equation}\label{dlambda}
  \abs{\lambda}:=\textrm{the unique $\b_k$-dominant element in the $\oW_{\g_k}$-orbit of $\lambda$,}
\end{equation}
where $\oW_{\g_k}$ denotes the Weyl group of $\g_k$ with respect to the Cartan subalgebra $\t_k$.

\subsection{Multiplicity one for classical branching rules}
Recall that $n\geq 2$, and the group $G=\GL_n(\R)\times \GL_{n-1}(\R)$ has a maximal compact subgroup
$ K=\oO(n)\times \oO(n-1)$.
Their complexified Lie algebras
\[
  \g=\g_n\times \g_{n-1}   \quad \textrm{and}\quad  \k=\o_n\times \o_{n-1}
\]
have respective Borel subalgebras
\[
  \b:=\b_n\times \b_{n-1}=\t\ltimes \n \quad \textrm{and}\quad  \b^\mathrm c:=\b_n^\mathrm c\times \b^\mathrm c_{n-1}=\t^{\mathrm c}\ltimes \n^\mathrm c,
\]
where
\[
 \t:=\t_n\times \t_{n-1},\,\, \n:=\n_n\times \n_{n-1}\quad \textrm{and}\quad \t^\mathrm c:=\t^\mathrm c_n\times \t^\mathrm c_{n-1},\,\,
  \n^\mathrm c:=\n^\mathrm c_n\times \n^\mathrm c_{n-1}.
\]
As in the Introduction, $\h=\g_{n-1}$ is diagonally embedded in $\g$, and $\c=\o_{n-1}$ is diagonally embedded in $\k$.

\begin{lem}\label{decg}
There are vector space decompositions
\begin{equation}\label{decg1}
  \g=\h\oplus \b\quad \textrm{and}\quad \k=\c\oplus \b^\mathrm c.
\end{equation}
\end{lem}
\begin{proof}
The first equality is proved by dimension counting and by writing down $\b_n$ and $\b_{n-1}$ explicitly. We omit the details. The second one is a consequence of the first one since both $\h$ and $\b$ are stable under the Cartan involution
\[
 \theta_n\times \theta_{n-1}: \g\rightarrow \g,\quad (x,y)\mapsto (-x^{\mathrm t}, -y^{\mathrm t}).
\]

\end{proof}

As usual, a superscript group (or Lie algebra) indicates the space of invariants of a group (or Lie algebra) representation; and ``$\mathcal{U}$" indicates the universal enveloping algebra of a complex Lie algebra.

Lemma \ref{decg} implies the following classical multiplicity one theorem.
\begin{lem}\label{mulone}
Let $F$ be an irreducible finite dimensional $\g$-module, and let $\chi$ be a one dimensional $\h$-module. Then every nonzero element $\phi\in \Hom_{\h}(F,\chi)$ does not vanish on $F^{\bar \n}$. Consequently,
 \[
  \dim \Hom_{\h}(F,\chi)\leq 1.
\]
\end{lem}
\begin{proof}
Lemma \ref{decg} implies that
\[
 \g=\h\oplus \bar \b.
\]
The first assertion of the Lemma then follows from the equality
\[
  F=\mathcal{U}(\g). F^{\bar \n}=(\mathcal{U}(\h)\,\mathcal{U}(\bar \b)). F^{\bar \n}=\mathcal{U}(\h). F^{\bar \n}.
\]

Since $\dim F^{\bar \n}=1$, the first assertion implies the second one.
\end{proof}

Note that $K^\circ=\SO(n)\times \SO(n-1)$, and $C^\circ=\SO(n-1)$ is diagonally embedded in $K^\circ$. The following lemma is proved as in Lemma \ref{mulone}.

\begin{lem}\label{mulone2}
Let $\tau$ be an irreducible representation of $K^\circ$. Then every nonzero $C^\circ$-invariant linear functional on $\tau$ does not vanish on $\tau^{\overline{\n^{\mathrm c}}}$. Consequently,
 \[
  \dim \tau^{C^\circ}\leq 1.
\]
\end{lem}

\subsection{The representation $F_\xi$}\label{fxi}

Let $F_\xi=F_\mu\otimes F_\nu$ and $j\in \Z$ be as in the Introduction. Whenever an irreducible representation of a compact Lie group occurs with  multiplicity one in another representation, we view the irreducible representation as a subspace of the other representation.

View $\mu$ as an element of $\t_n^*$, and denote by $\tau_{\mu}$ the irreducible representation of $\SO(n)$ of highest weight $[\abs{\mu}]\in {\t_n^\mathrm c}^*$ (see \eqref{dlambda}). Then $\tau_\mu$ occurs with multiplicity one in $F_\mu$, and  $\tau_{\mu}\supset (F_\mu)^{\n_n}$. Taking dual representations, we get the following result.
\begin{lem}\label{lweight}
The representation  $\tau_{\mu}^\vee$ occurs with multiplicity one in $F_\mu^\vee$. It contains the one dimensional space $(F_\mu^\vee)^{\overline{\n_n}}$.
\end{lem}

Define  $\tau_\nu$ similarly. Then
\[
   \tau_\xi:=\tau_\mu\otimes \tau_\nu
\]
is an irreducible representation of $K^\circ$. Lemma \ref{lweight} (and its analog for $\nu$) and Lemma \ref{mulone} imply the following lemma.

\begin{lem}\label{nv1} The representation  $\tau_\xi^\vee$ occurs with multiplicity one in $F_\xi^\vee$.
Moreover, every nonzero element of $\Hom_{H_\C}(F_\xi^\vee, {\det}^{j})$ does not vanish on $\tau_\xi^\vee\subset F_\xi^\vee$.
\end{lem}

Denote by $\tau_{-\mu}$ the irreducible representation of $\SO(n)$ of highest weight $[\abs{-\mu}]\in {\t_n^\mathrm c}^*$. (Indeed, $\tau_{-\mu}\cong \tau_\mu$, but we shall not use this fact as it does not generalize to the complex case of Section \ref{complexg}.) Then $\tau_{-\mu}$ occurs with multiplicity one in $F_\mu^\vee$, and  $\tau_{-\mu}\supset (F_\mu^\vee)^{\n_n}$. Taking dual representations, we get the following result as in Lemma \ref{lweight}.
\begin{lem}\label{lweight2}
The representation  $\tau_{-\mu}^\vee$ occurs with multiplicity one in $F_\mu$. It contains the one dimensional space  $F_\mu^{\overline{\n_n}}$.
\end{lem}

Define $\tau_{-\nu}$ similarly and put
\[
   \tau_{-\xi}:=\tau_{-\mu}\otimes \tau_{-\nu}.
\]
Similar to Lemma \ref{nv1}, we have the following result.
\begin{lem}\label{nv10} The representation  $\tau_{-\xi}^\vee$ occurs with multiplicity one in $F_\xi$.
Moreover, every nonzero element of $\Hom_{H_\C}(F_\xi, {\det}^{-j})$ does not vanish on $\tau_{-\xi}^\vee\subset F_\xi$.
\end{lem}

In particular, Lemma \ref{nv1} and Lemma \ref{nv10} imply  that
\begin{equation}\label{lxi1}
  (\tau_\xi)^{C^\circ}\neq 0\qquad\textrm{and}\qquad (\tau_{-\xi})^{C^\circ}\neq 0.
\end{equation}

\subsection{The space $\wedge^{b_n+b_{n-1}}(\g/\tilde \k)$}\label{fxi2}

Denote by $2\rho_n\in \t_n^*$ the sum of all weights of $\n_n$, and by $2\rho_n^{\mathrm c}\in {\t_n^{\mathrm c}}^*$ the sum of all weights of $\n_n^\mathrm c$.
Write $\tau_n$ for the irreducible representation of $\SO(n)$ of highest weight $[2\rho_n]-2\rho_n^\mathrm c\in {\t_n^{\mathrm c}}^*$. The following lemma is easily verified and we omit the details.
\begin{lem}\label{ngln1}
The representation $\tau_n$ occurs with multiplicity one in $\wedge^{b_n}(\g_n/\g\o_n)$. It contains the one dimensional space $\wedge^{b_n}(\n_n/\n_n^{\mathrm c})$.
\end{lem}

Define  $\tau_{n-1}$ similarly. Then
\[
   \tau_{n,n-1}:=\tau_n\otimes \tau_{n-1}
\]
is an irreducible representation of $K^\circ$. Note that $\tilde \k=\g\o_n\times \g\o_{n-1}$. Similar to Lemma \ref{ngln1}, we have the following lemma.
\begin{lem}\label{ngln2}
The representation $\tau_{n,n-1}$ occurs with multiplicity one in $\wedge^{b_n+b_{n-1}}(\g/\tilde \k)$. It contains the one dimensional space $\wedge^{b_n+b_{n-1}}(\n/\n^\mathrm c)$.
\end{lem}

Using Lemma \ref{ngln2}, we fix a nonzero element
\begin{equation}\label{etan}
   \eta_{n,n-1}\in \Hom_{K^\circ}(\wedge^{b_n+b_{n-1}}(\g/\tilde \k), \tau_{n,n-1}).
\end{equation}
Write
\[
  \iota_{n,n-1}: \wedge^{b_n+b_{n-1}}(\h/\c)\rightarrow \wedge^{b_n+b_{n-1}}(\g/\tilde \k)
\]
for the natural embedding.
\begin{lem}\label{nv2}
The composition
\begin{equation}\label{taueta}
  \wedge^{b_n+b_{n-1}}(\h/\c)\xrightarrow{\iota_{n,n-1}} \wedge^{b_n+b_{n-1}}(\g/\tilde \k)\xrightarrow{\eta_{n,n-1}}  \tau_{n,n-1}
\end{equation}
is nonzero. Its image is equal to $(\tau_{n,n-1})^{C^\circ}$.
\end{lem}
\begin{proof}
Fix a $K$-invariant positive definite Hermitian form $\la\,,\ra$ on $\g/\tilde \k$. This induces a $K$-invariant positive definite Hermitian form $\la\,,\ra_\wedge$ on $\wedge^{b_n+b_{n-1}}(\g/\tilde \k)$. Note that
\[
\eta_{n,n-1}: \wedge^{b_n+b_{n-1}}(\g/\tilde \k)\rightarrow \tau_{n,n-1}
\]
is a scalar multiple of the orthogonal projection. By Lemma \ref{ngln2}, in order to prove the first assertion of the lemma,
it suffices to show that the one dimensional spaces
$\wedge^{b_n+b_{n-1}}(\h/\c)$ and $\wedge^{b_n+b_{n-1}}(\n/\n_\mathrm c)$
are not perpendicular to each other under the form
$\la\,,\,\ra_{\wedge}$, or equivalently, the paring
\[
  \la\,,\,\ra: \h/\c\times \n/\n^{\mathrm c}\rightarrow  \C
\]
is nondegenerate.

Note that the orthogonal complement of $\n/\n^{\mathrm c}$ in $\g/\tilde \k$ is $\bar \b/(\bar \b\cap \tilde \k)$. It follows from Lemma \ref{decg} that
\[
  (\h/\c)\cap (\bar \b/(\bar \b\cap \tilde \k))=0.
\]
This proves the first assertion of the lemma.

The image of the composition \eqref{taueta} is a nonzero subspace of  $(\tau_{n,n-1})^{C^\circ}$.  By Lemma \ref{mulone2}, the latter space is at most one dimensional. Therefore the second assertion follows.
\end{proof}

In particular, Lemma \ref{nv2} implies that
\begin{equation}\label{taunn}
  (\tau_{n,n-1})^{C^\circ}\neq 0.
\end{equation}

\subsection{Cartan products and PRV components}\label{cpprv}
Let $R$ be a connected compact Lie group. In this subsection, we review some general results about tensor products of irreducible representations of $R$.
Let $\sigma_1$ and $\sigma_2$ be two irreducible representations of $R$.  Fix a Cartan subgroup of $R$ and fix a positive system of the associated root system. Respectively write $\lambda^+_i$ and $\lambda_i^-$ for the highest and lowest weights of $\sigma_i$ ($i=1,2$).

Write $\sigma_3$ for the irreducible representation of $R$ of highest weight $\lambda_1^++\lambda_2^+$ (or equivalently, of lowest weight $\lambda_1^- +\lambda_2^-$). The following lemma is obvious.
\begin{lem}\label{cartanc}
The representation $\sigma_3$ occurs with  multiplicity one in $\sigma_1\otimes \sigma_2$. It contains all tensor products of lowest weight vectors in $\sigma_1$ and $\sigma_2$.
\end{lem}

The representation $\sigma_3$ is called the Cartan component of $\sigma_1\otimes \sigma_2$, or the Cartan product of $\sigma_1$ and $\sigma_2$. The following lemma is known (see \cite[Section 2.1]{Ya}).
\begin{lem}\label{cprod}
Let $f: \sigma_1\otimes \sigma_2\rightarrow \sigma_3$ be a nonzero $R$-equivariant linear map. Then $f$ maps all nonzero decomposable vectors (namely, vectors of the form $u\otimes v\in \sigma_1\otimes \sigma_2$) to nonzero vectors.
\end{lem}

Note that the weights $\lambda_1^+ +\lambda_2^-$ and $\lambda_1^-+\lambda_2^+$ stay in the same orbit under the Weyl group action. Write $\sigma_4$ for the irreducible representation of $R$ with extremal weights $\lambda_1^+ +\lambda_2^-$ and $\lambda_1^-+\lambda_2^+$.

\begin{lem}\cite[Corollary
1 to Theorem 2.1]{PRV}
The representation $\sigma_4$ occurs with  multiplicity one in $\sigma_1\otimes \sigma_2$.
\end{lem}

The representation $\sigma_4$ is called the PRV component of $\sigma_1\otimes \sigma_2$.

\begin{lem}\label{prvc}
The PRV component of $\sigma_1^\vee\otimes \sigma_3$ is isomorphic to $\sigma_2$.
\end{lem}
\begin{proof}
The lemma follows by noting that the lowest weight of $\sigma_1^\vee$ is $-\lambda_1^+$.
\end{proof}

It is obvious that the arguments of this subsection also apply to finite dimensional algebraic representations of connected reductive complex linear algebraic groups.

\subsection{PRV components and classical branching rules}
Now we return to the setting before Section \ref{cpprv}. Let $\sigma_1$ and $\sigma_2$ be two irreducible representations of $K^\circ$ and write $\sigma_3$ for their Cartan product. Assume that
\[
  (\sigma_i)^{C^\circ}\neq 0,  \qquad \textrm{for } i=1,2.
\]

\begin{lem}\label{cartann}
The space $(\sigma_3)^{C^\circ}$ is nonzero.
\end{lem}
\begin{proof}
This is a direct consequence of Lemmas \ref{mulone2} and \ref{cartanc}.
\end{proof}

\begin{prpl}\label{nprv}
Every nonzero $C^\circ \times C^\circ$-invariant linear functional on $\sigma_1^\vee\otimes \sigma_3$ does not vanish on the PRV component
\[
  \sigma_2\subset \sigma_1^\vee\otimes \sigma_3.
\]
\end{prpl}
 \begin{proof}
 Write $\sigma_i=\alpha_i\otimes \beta_i$, where $\alpha_i$ is an irreducible representation of $\SO(n)$, and $\beta_i$ is an irreducible representations of $\SO(n-1)$ ($i=1,2,3$). Then $\alpha_3$ is the Cartan product of $\alpha_1$ and $\alpha_2$, and $\beta_3$ is the Cartan product of $\beta_1$ and $\beta_2$.

By Lemma \ref{mulone2}, every nonzero $C^\circ \times C^\circ$-invariant linear functional on $\sigma_1^\vee\otimes \sigma_3$ is of the form
$\phi_1\otimes \phi_3$, where $\phi_1$ is a nonzero $C^\circ$-invariant linear functional on $\sigma_1^\vee$, and $\phi_3$ is a  nonzero $C^\circ$-invariant linear functional on $\sigma_3$.

Fix a generator $v_2\in (\alpha_2\otimes \beta_2)^{C^\circ}$. It is routine to  check that the following diagram commutes:
 \[
 \begin{CD}
            \Hom_{\SO(n)\times \SO(n-1)}(\alpha_2\otimes \beta_2, \alpha_1^\vee\otimes \beta_1^\vee\otimes \alpha_3\otimes \beta_3) @>f\mapsto ((\phi_1\otimes \phi_3)\circ f)(v_2) >> \C \\
            @VV \cong V           @VV = V\\
           \Hom_{\SO(n-1)}(\beta_3^\vee, \beta_1^\vee\otimes \beta_2^\vee)\otimes \Hom_{\SO(n)}(\alpha_1\otimes \alpha_2, \alpha_3)  @>f'\otimes f\mapsto \phi_3(f\circ (\phi_1\otimes v_2)\circ f')>> \C,\\
  \end{CD}
\]
where the left vertical arrow is the canonical isomorphism, and in the bottom horizontal arrow, we view
\begin{eqnarray*}
   &&\phi_1\in \Hom_{\SO(n-1)}(\beta_1^\vee,\alpha_1)=\Hom_{C^\circ}(\alpha_1^\vee\otimes \beta_1^\vee, \C), \smallskip \\
   && v_2\in \Hom_{\SO(n-1)}(\beta_2^\vee,\alpha_2)=(\alpha_2\otimes \beta_2)^{C^\circ},\,\,\textrm{and}\smallskip  \\
   && f\circ (\phi_1\otimes v_2)\circ f'\in \Hom_{\SO(n-1)}(\beta_3^\vee,\alpha_3)=(\alpha_3\otimes \beta_3)^{C^\circ}.
\end{eqnarray*}

In order to prove the proposition, it suffices to show that the top horizontal arrow of the diagram is nonzero, or equivalently, it suffices to show that the bottom horizontal arrow is nonzero.

Note that $\beta_3^\vee$ is the Cartan product of $\beta_1^\vee$ and $\beta_2^\vee$. Pick a generator
\[
  f'_0\otimes f_0\in  \Hom_{\SO(n-1)}(\beta_3^\vee, \beta_1^\vee\otimes \beta_2^\vee)\otimes \Hom_{\SO(n)}(\alpha_1\otimes \alpha_2, \alpha_3).
\]
Let $u_3'$ be a nonzero lowest weight vector in $\beta_3^\vee$.  By Lemma \ref{cartanc}, $f_0'(u_3')$ is a nonzero decomposable vector in $\beta_1^\vee\otimes \beta_2^\vee$. Consequently, $((\phi_1\otimes v_2)\circ f_0')(u_3')$ is a nonzero decomposable vector in $\alpha_1\otimes \alpha_2$. Then Lemma \ref{cprod} implies that
\[
  (f_0\circ(\phi_1\otimes v_2)\circ f_0')(u_3')\neq 0,
\]
and consequently, $f_0\circ(\phi_1\otimes v_2)\circ f_0'$ is a generator of the one dimensional space
\[
\Hom_{\SO(n-1)}(\beta_3^\vee,\alpha_3)=(\alpha_3\otimes \beta_3)^{C^\circ}.
\]
This shows that the bottom horizontal arrow is nonzero since $\phi_3$ does not vanish on $(\alpha_3\otimes \beta_3)^{C^\circ}$.
 \end{proof}

\section{Cohomological representations}\label{cohom}

\subsection{Cohomological inductions}\label{coh0}

Recall from Section \ref{str} the Borel subalgebra
\[
  \b_n=\t_n\ltimes \n_n\subset \g_n.
\]
Write $\oT_n(\C)$ for the Cartan subgroup of $\GL_n(\C)$ with Lie algebra $\t_n$. Recall that $2\rho_n\in \t_n^*$ denotes the sum of all weights of $\n_n$. Let $\mu$ be as in the Introduction, which is assumed to be pure as in \eqref{puremu}. Write $\C_{\abs{\mu}+2\rho_n}$ for the one dimensional algebraic representation of $\oT_n(\C)$ with weight $\abs{\mu}+2\rho_n\in \t_n^*$  (see \eqref{dlambda}).  We also view it as a $\overline{\b_n}$-module  via the quotient map
\[
  \overline{\b_n}=\t_n\ltimes \overline{\n_n}\rightarrow \t_n.
\]

Put
\[
  \oT_n^\mathrm c:=\oO(n)\cap \oT_n(\C).
\]
Then
\begin{equation}\label{vmu}
  V_\mu:=\mathcal{U}(\g_n)\otimes_{\mathcal{U}(\overline{\b_n} )} \C_{\abs{\mu}+2\rho_n}
  \end{equation}
  is a $(\g_n, \oT_n^{\mathrm c})$-module, where $\g_n$ acts by left multiplications,
and $\oT_n^{\mathrm c}$ acts by the tensor product of its adjoint action on
$\mathcal{U}(\g_n)$ and the restriction of the $\oT_n(\C)$-action on $\C_{\abs{\mu}+2\rho_n}$. By \cite[Theorem 7.6.24]{Di} or \cite[Corollary 5.105]{KV}, we know that $V_\mu$ is irreducible as a $\g_n$-module.

Denote by $\Pi$ the Bernstein functor (see \cite[Page 196]{KV}) from the category of $(\g_n, \oT_n^{\mathrm c})$-modules to the category of $(\g_n, \oO(n))$-modules. Recall that for every $(\g_n, \oT_n^{\mathrm c})$-module $M$,
\be\label{functorpi}
  \Pi(M)=\mathcal{H}(\g_n, \oO(n))\otimes_{\mathcal{H}(\g_n,  \oT_n^{\mathrm c})} M,
\ee
where $\mathcal H$ indicates the Hecke algebra. The reader is referred to \cite[Chapter I]{KV} for details on Hecke algebras. Write $\Pi_i$ for the $i$-th left derived functor of $\Pi$ ($i\in \Z$). Then by \cite[Theorems 5.35 and 5.99]{KV},
\[
  \Pi_i(V_\mu)=0\quad\textrm{unless}\quad i=S_n:=\dim \n_n^{\mathrm c}=\left\{
         \begin{array}{ll}
           \frac{n(n-2)}{4}, & \hbox{if $n$ is even;} \\
           \frac{(n-1)^2}{4}, & \hbox{if $n$ is odd,}
         \end{array}
       \right.
\]
and by \cite[Corollary 8.28]{KV},
\[
 W_\mu:=\Pi_{S_n}(V_\mu)
\]
is an irreducible $(\g_n, \oO(n))$-module.

Denote by $W_\mu^\infty$ the Casselman-Wallach smooth globalization of $W_\mu$, namely, it is the Casselman-Wallach representation of $\GL_n(\R)$ whose space of $\oO(n)$-finite vectors equals $W_\mu$ as a $(\g_n, \oO(n))$-module.

\begin{lem}\label{temp}
The representation $W_\mu^\infty|_{\SL_n^{\pm}(\R)}$ is unitarizable and tempered.
\end{lem}
\begin{proof}
See \cite[Theorem 6.8.1]{Wa1}, or \cite[Theorem 9.1 and Corollary 11.229]{KV}. Unitarizability of the derived functor modules is first proved by Vogan in \cite{V}, and a simpler proof is given by Wallach in \cite{Wa3}.
\end{proof}

\subsection{Restricting to $\GL_n^+(\R)$} Denote by $\GL_n^+(\R)$ the subgroup of $\GL_n(\R)$ consisting of matrices of positive determinants. Then $\SO(n)$ is a maximal compact subgroup of it. Write
 \[
   V_\mu^\circ:=V_\mu=\mathcal{U}(\g_n)\otimes_{\mathcal{U}(\overline{\b_n} )} \C_{\abs{\mu}+2\rho_n},
 \]
 to be viewed as a $(\g_n, {\oT_n^{\mathrm c}}^\circ)$-module.

 Denote by $\Pi^{\circ}$ the Bernstein functor from the category of $(\g_n, {\oT_n^{\mathrm c}}^\circ)$-modules to the category of $(\g_n, \SO(n))$-modules. Write $\Pi^{\circ}_i$ for its $i$-th left derived functor ($i\in \Z$). As in Section \ref{coh0},
 \[
  \Pi_i^\circ(V_\mu^\circ)=0\quad\textrm{unless}\quad i=S_n,
\]
and
\[
 W_\mu^\circ:=\Pi^\circ_{S_n}(V^\circ_\mu)
\]
is an irreducible $(\g_n, \SO(n))$-module.

Since
\[
  \mathcal{H}(\g_n, \SO(n))\subset \mathcal{H}(\g_n, \oO(n)) \quad\textrm{and}\quad \mathcal H(\overline{\b_n}, {\oT_n^{\mathrm c}}^\circ)\subset\mathcal{H}(\overline{\b_n}, \oT_n^{\mathrm c}),
\]
by passing to homology, the identity map
\[
  V^\circ_\mu\rightarrow V_\mu
\]
induces a $(\g_n, \SO(n))$-module homomorphism
\begin{equation}\label{glp}
   W^\circ_\mu\rightarrow W_\mu.
\end{equation}

\begin{lem}\label{rglp}
If $n$ is odd, then \eqref{glp} is an isomorphism. If $n$ is even, then \eqref{glp} is injective and induces a $(\g_n, \oO(n))$-module isomorphism
\begin{equation}\label{glp2}
  \mathcal{H}(\g_n, \oO(n)) \otimes_{\mathcal{H}(\g_n, \SO(n)) }  W^\circ_\mu\cong  W_\mu.
\end{equation}
\end{lem}
 \begin{proof}
 If $n$ is odd, then \eqref{glp} is an isomorphism since
 \[
   \mathcal{H}(\g_n, \oO(n))=\mathcal{H}(\g_n, \SO(n))\otimes \mathcal H(\{\pm 1\})
 \]
 and
 \[
  \mathcal H(\overline{\b_n}, \oT_n^{\mathrm c})=\mathcal{H}(\overline{\b_n}, {\oT_n^{\mathrm c}}^\circ)\otimes \mathcal H(\{\pm 1\}).
 \]
 If $n$ is even, then \eqref{glp2} holds by induction by steps since ${\oT_n^{\mathrm c}}^\circ=\oT_n^{\mathrm c}$.

 \end{proof}

 In all cases, we view $W_\mu^\circ$ as a $(\g_n, \SO(n))$-submodule of $W_\mu$ through the embedding \eqref{glp}.

\subsection{Bottom layers}\label{bl}

 Denote by $\Pi^{\mathrm c}$ the Bernstein functor from the category of $(\o_n, {\oT_n^{\mathrm c}}^\circ)$-modules to the category of $(\o_n, \SO(n))$-modules. Write $\Pi^{\mathrm c}_i$ for its $i$-th left derived functor ($i\in \Z$). Similar to $V_\mu$, we have an $(\o_n, {\oT_n^{\mathrm c}}^\circ)$-module
\[
  V^{\mathrm c}_\mu:=\mathcal{U}(\o_n)\otimes_{\mathcal{U}(\overline{\b_n^{\mathrm c}} )} \C_{[\abs{\mu}+2\rho_n]}.
  \]
Recall the irreducible representations $\tau_{\mu}$ and $\tau_n$ of $\SO(n)$ from Section \ref{fxi} and Section \ref{fxi2}, respectively. Define
 \[
    \tau_{\mu}^+:=\textrm{the Cartan product of $\tau_{\mu}$ and $\tau_n$.}
      \]
Then the algebraic version of the Borel-Bott-Weil theorem \cite[Corollary 4.160]{KV} implies that
\[
  \Pi_i^{\mathrm c}(V^{\mathrm c}_\mu)\cong\left\{
         \begin{array}{ll}
           0, & \hbox{if $i\neq S_n$;} \\
           \tau_\mu^+, & \hbox{if $i=S_n$,}
         \end{array}
       \right.
\]
Put
\[
 W^{\mathrm c}_\mu:=\Pi^{\mathrm c}_{S_n}(V^{\mathrm c}_\mu)\cong \tau_\mu^+.
\]

Since
\[
  \mathcal{H}(\o_n, \SO(n))\subset \mathcal{H}(\g_n, \SO(n)) \quad\textrm{and}\quad \mathcal H(\overline{\b_n^{\mathrm c}}, {\oT_n^{\mathrm c}}^\circ)\subset\mathcal{H}(\overline{\b_n}, {\oT_n^{\mathrm c}}^\circ),
\]
by passing to homology, the inclusion
\[
  \beta_\mu: V^{\mathrm c}_\mu \rightarrow V_\mu^\circ
\]
 induces an $\SO(n)$-equivariant linear map
\[
  \mathcal B_\mu: W^{\mathrm c}_\mu\rightarrow W_\mu^\circ.
\]
The map $\mathcal B_\mu$ is injective and is called the Bottom layer map (see \cite[Section V.6]{KV}). By \cite[Theorem 5.80]{KV}, one knows that $\tau_\mu^+$ has multiplicity one in $W_\mu^\circ$, and by \cite[Proposition 10.24]{KV}, it is the unique minimal $\SO(n)$-type of $W_\mu^\circ$ (see \cite[Section X.2]{KV} for details on ``minimal $K$-types"). When $n$ is even, the unique minimal $\oO(n)$-type of $W_\mu$ is the irreducible representation
\begin{equation}\label{mo}
  \mathcal H(\oO(n))\otimes_{\mathcal H(\SO(n))} \tau_\mu^+.
\end{equation}
When restricted to $\SO(n)$,  \eqref{mo} is the direct sum $\tau_\mu^+\oplus ({\tau_\mu^+})'$ of two inequivalent irreducible representations of $\SO(n)$, where $({\tau_\mu^+})'$ denotes the twist of the representation $\tau_\mu^+$ by an outer automorphism $x\mapsto gxg^{-1}$ of $\SO(n)$, with $g\in \oO(n)\setminus \SO(n)$.

Together with Lemma \ref{rglp}, the above argument implies the following lemma.
\begin{lem}\label{mink}
The irreducible representation $\tau_\mu^+$ of $\SO(n)$ occurs with multiplicity one in $W_\mu$.
\end{lem}

\subsection{Nonvanishing of cohomology}\label{nonvc}

Note that the center $\R^\times$ of $\GL_n(\R)$ acts trivially on $F_\mu^\vee\otimes W_\mu^\infty$. Therefore
\begin{eqnarray*}
   && \oH^{b_n}(\g_n,\GO(n)^\circ; F_\mu^\vee\otimes W_\mu^\infty) \\
   &=& \oH^{b_n}(\sl_{n}(\C),\SO(n); F_\mu^\vee\otimes W_\mu^\infty)  \\
   &=& \Hom_{\SO(n)}(\wedge^{b_n}(\s\l_n(\C)/\s\o_n(\C)),  F_\mu^\vee\otimes W_\mu^\infty)  \qquad \textrm{by \cite[Proposition 9.4.3]{Wa1}} \\
   &=& \Hom_{\SO(n)}(\wedge^{b_n}(\g_n/\g\o_n),  F_\mu^\vee \otimes W_\mu^\infty) \\
             &\supset & \Hom_{\SO(n)}(\wedge^{b_n}(\g_n/\g\o_n), \tau_{\mu}^\vee\otimes \tau_\mu^+)\qquad \textrm{by Lemma \ref{lweight} and Lemma \ref{mink}}\\
        &\supset & \Hom_{\SO(n)}(\tau_n, \tau_{\mu}^\vee\otimes \tau_\mu^+) \qquad \textrm{ by  Lemma \ref{ngln1}}.
\end{eqnarray*}
The last hom space is one dimensional by Lemma \ref{prvc}. Together with Lemma \ref{temp}, this shows that $W_\mu^\infty$ is a representation in $\Omega(\mu)$. Consequently,
\begin{equation}\label{omegamu}
  \Omega(\mu)=\left\{
                \begin{array}{ll}
                  \{W_\mu^\infty\}, & \hbox{if $n$ is even;} \smallskip \\
                \{W_\mu^\infty, W_\mu^\infty\otimes \sgn\}, & \hbox{if $n$ is odd.}
                \end{array}
              \right.
\end{equation}

\subsection{Translations}

Put
\begin{equation}\label{mut}
  \tilde \mu:=(\mu_1-\mu_n, \mu_2-\mu_{n-1},\cdots, \mu_n-\mu_1)
\end{equation}
as in the Introduction so that $F_{\tilde \mu}$ is the Cartan product of $F_\mu$ and $F_\mu^\vee$. Applying  the previous argument to $\tilde \mu$, we get a representation $\C_{\abs{\tilde \mu}+2\rho_n}$ of $\oT_n(\C)$, and injective maps
\[
  \beta_{\tilde \mu}: V^{\mathrm c}_{\tilde \mu}\rightarrow V_{\tilde \mu}^\circ= V_{\tilde \mu}\quad\textrm{and}\quad \mathcal B_{\tilde \mu}: \tau_{\tilde \mu}^+\cong W^{\mathrm c}_{\tilde \mu}\rightarrow W_{\tilde \mu}^\circ\subset W_{\tilde \mu}.
\]

\begin{lem}\label{kappa}
Up to scalar multiplication, there is a unique nonzero $(\g_n, \oT_n^{\mathrm c})$-module homomorphism
\[
  \kappa: V_{\mu}\rightarrow F_\mu\otimes V_{\tilde \mu}.
\]
Moreover, $\kappa$ is injective and its image is a direct summand as a $(\g_n, \oT_n^{\mathrm c})$-submodule of $F_\mu\otimes V_{\tilde \mu}$.
\end{lem}
\begin{proof}
Note that both $\abs{\mu}+\rho_n$ and $\abs{\tilde \mu}+\rho_n$ are strictly dominant with respect to $\b_n$. The proof of Theorem 7.237 in \cite[Pages 545-546]{KV} shows that  $V_{\mu}$ is a direct summand of $F_\mu\otimes V_{\tilde \mu}$, and it occurs with  multiplicity one in the composition series of $F_\mu\otimes V_{\tilde \mu}$. This implies the lemma.

\end{proof}

The injective homomorphism $\kappa$ of Lemma  \ref{kappa} is given as follows:
\begin{eqnarray*}
  V_\mu &=& \mathcal{U}(\g_n)\otimes_{\mathcal{U}(\overline{\b_n} )} \C_{\abs{\mu}+2\rho_n} \\
   &\cong & \mathcal{U}(\g_n)\otimes_{\mathcal{U}(\overline{\b_n} )}( F_\mu^{\overline{\n_n}}\otimes \C_{\abs{\tilde \mu}+2\rho_n})\\
    &\hookrightarrow & \mathcal{U}(\g_n)\otimes_{\mathcal{U}(\overline{\b_n} )}( F_\mu\otimes \C_{\abs{\tilde \mu}+2\rho_n})\\
   &\stackrel{\textrm{Mackey isomorphism}}{\cong} &  F_\mu\otimes (\mathcal{U}(\g_n)\otimes_{\mathcal{U}(\overline{\b_n} )} \C_{\abs{\tilde \mu}+2\rho_n})\\
   &=& F_\mu\otimes V_{\tilde \mu}.
\end{eqnarray*}
The reader is referred to \cite[Theorem 2.103]{KV} for  the Mackey isomorphism.

Recall form Lemma \ref{lweight2} that $\tau_{-\mu}^\vee$ occurs in $F_\mu$ with multiplicity one. Write
$\iota_\mu: \tau_{-\mu}^\vee\rightarrow F_\mu$ for the inclusion map. Let $\kappa$ be as in Lemma \ref{kappa}. Similar to Lemma \ref{kappa}, we have an  $(\o_n, {\oT_n^{\mathrm c}}^\circ)$-module homomorphism
\[
  \kappa^{\mathrm c}: V_{\mu}^\mathrm{c}\rightarrow \tau_{-\mu}^\vee\otimes V^{\mathrm c}_{\tilde \mu},
\]
which is suitably normalized so that  the diagram
 \begin{equation}\label{cdk0}
 \begin{CD}
            V_\mu @>\kappa >> F_\mu\otimes V_{\tilde \mu} \\
            @A A \beta_\mu A           @AA\iota_\mu\otimes \beta_{\tilde \mu} A\\
           V_\mu^{\mathrm c} @>\kappa^{\mathrm c}>> \tau_{-\mu}^\vee\otimes V^{\mathrm c}_{\tilde \mu}\\
  \end{CD}
\end{equation}
commutes.

Recall the functor $\Pi$ of \eqref{functorpi}. We have a natural isomorphism (the Mackey isomorphism)
\be\label{mackey0}
  \Pi(F_\mu\otimes \,\cdot\,) \cong F_\mu\otimes\Pi(\,\cdot\,)
\ee
between two functors from the category of $(\g_n, \oT_n^{\mathrm c})$-modules to the category of $(\g_n, \oO(n))$-modules.
Since the functor $F_\mu\otimes (\,\cdot\,)$ from the category of $(\g_n, \oT_n^{\mathrm c})$-modules to itself is exact and maps projective objects to projective objects, the isomorphism \eqref{mackey0} induces an isomorphism
\be\label{mackey1}
   \Pi_j(F_\mu\otimes M) \stackrel{\textrm{Mackey}}{\cong} F_\mu\otimes\Pi_j(M),
\ee
for all $(\g_n, \oT_n^{\mathrm c})$-module $M$ and all $j\in \Z$. Similarly, we have an isomorphism
\be\label{mackey22}
   \Pi_j^{\mathrm c}(\tau_{-\mu}^\vee\otimes M^{\mathrm c}) \stackrel{\textrm{Mackey}}{\cong} \tau_{-\mu}^\vee\otimes\Pi_j(M^{\mathrm c}),
\ee
for all $(\o_n, {\oT_n^{\mathrm c}}^\circ)$-module $M^{\mathrm c}$ and all $j\in \Z$.

Applying the derived Bernstein functors to the diagram \eqref{cdk0}, and using the derived Mackey isomorphisms \eqref{mackey1} and \eqref{mackey22}, we get a commutative diagram
\[
 \begin{CD}
          W_\mu=\Pi_{S_n}(V_\mu) @>\Pi_{S_n}(\kappa)>> \Pi_{S_n}(F_\mu\otimes V_{\tilde \mu}) @>\textrm{Mackey} >\cong > F_\mu\otimes\Pi_{S_n}(V_{\tilde \mu})=F_\mu\otimes W_{\tilde \mu}\\
            @AA \mathcal B_\mu A           @AAA  @AA \iota_\mu\otimes \mathcal B_{\tilde \mu} A\\
          \tau_\mu^+\cong \Pi_{S_n}^{\mathrm c}(V_\mu^{\mathrm c}) @>\Pi_{S_n}^{\mathrm c}(\kappa^{\mathrm c})>> \Pi_{S_n}^{\mathrm c}(\tau_{-\mu}^\vee\otimes V^{\mathrm c}_{\tilde \mu}) @> \textrm{Mackey} >\cong>
  \tau_{-\mu}^\vee \otimes \Pi_{S_n}^{\mathrm c}(V^{\mathrm c}_{\tilde \mu})\cong \tau_{-\mu}^\vee\otimes \tau_{\tilde \mu}^+. \\
  \end{CD}
\]
Recall from Lemma \ref{kappa} that $\kappa$ is injective and its image is a direct summand in its range. This  implies that the map $\Pi_{S_n}(\kappa)$ is injective. Likewise, $\Pi^\mathrm{c}_{S_n}(\kappa^\mathrm{c})$ is also injective.

Note that by Lemma \ref{prvc}, $\tau_\mu^+$ is the PRV component of $\tau_{-\mu}^\vee\otimes \tau_{\tilde \mu}^+$.  In conclusion, we have proved the following proposition.
\begin{prpl}\label{transmu}
The irreducible $\SO(n)$-representation
 \[
   \tau_\mu^+\subset \tau_{-\mu}^\vee\otimes \tau_{\tilde \mu}^+
   \subset F_\mu\otimes W_{\tilde \mu}
 \]
 generates  an irreducible $(\g_n, \oO(n))$-submodule of  $F_\mu\otimes W_{\tilde \mu}$ which is  isomorphic to $W_\mu$.
\end{prpl}

Theorem 7.237 in \cite{KV} implies that $W_\mu$ is a direct summand of $F_\mu\otimes W_{\tilde \mu}$, and it occurs with multiplicity one in the composition series of $F_\mu\otimes W_{\tilde \mu}$. The point of Proposition \ref{transmu} is that it identifies the minimal $\oO(n)$-type of $W_\mu$  inside $F_\mu\otimes W_{\tilde \mu}$.

\section{Cohomological test vectors}\label{ctest}

\subsection{The result}\label{results}
Recall the weight $\nu$ from \eqref{intrnu}, which is assumed to be pure as before. Applying the discussion of Section \ref{cohom} to $\nu$, we get spaces
\[
  \tau_\nu^+\subset W_{\nu}\subset W_{\nu}^\infty.
  \]
Write
\[
 W_\xi^\infty:=W_\mu^\infty\widehat \otimes W_\nu^\infty.
\]
Then by Lemma \ref{mink} and its analog for $\nu$, the representation
\[
 \tau_\xi^+:=\tau_\mu^+\otimes \tau_\nu^+
\]
of $K^\circ$ occurs with multiplicity one in $W_\xi^\infty$.

Let $j$ be as in the Introduction. This section is devoted to a proof of the following proposition.
\begin{prpl}\label{test}
Every nonzero element of $\Hom_H(W_\xi^\infty, \abs{\det}^{-j})$ does not vanish on $\tau_\xi^+\subset W_\xi^\infty$.
\end{prpl}

\subsection{A special case}\label{secspecial}

Recall $\tilde \mu$ from \eqref{mut}. Define $\tilde \nu$ similarly. Similar to \eqref{f}, put
\[
   F_{\tilde \xi}:= F_{\tilde \mu}\otimes  F_{\tilde \nu}.
\]
As in Section \ref{results}, we have an irreducible representation
\[
 W_{\tilde \xi}^\infty:=W_{\tilde \mu}^\infty\widehat \otimes W_{\tilde \nu}^\infty
\]
of $G$, which contains the irreducible representation
\[
  \tau_{\tilde \xi}^+:=\tau_{\tilde \mu}^+\otimes \tau_{\tilde \nu}^+
  \]
 of $K^\circ$ with multiplicity one.

 Recall that
 \[
  \Hom_{H_\C}(F_\xi^\vee, {\det}^j) \neq 0,
  \]
and consequently,
\[
    \Hom_{H_\C}(F_\xi, {\det}^{-j}) \neq 0.
  \]
Note that $F_{\tilde \xi}^\vee\cong F_{\tilde \xi}$ is the Cartan product of $F_\xi$ and $F_\xi^\vee$. Similar to Lemma \ref{cartann}, we have that
\[
    \Hom_{H_\C}(F_{\tilde \xi}^\vee, {\det}^0) \neq 0.
  \]
Therefore $\tilde \mu$ and $\tilde \nu$ are compatible, and $\frac{1}{2}$ is a critical place for $W_{\tilde \mu}^\infty\times W_{\tilde \nu}^\infty$.

Note that the representation $W_{\tilde \xi}^\infty$ of $G$ has trivial central character. Therefore Lemma \ref{temp} (for $\tilde \mu$ and $\tilde \nu$) implies that $W_{\tilde \xi}^\infty$  is  unitarizable and tempered.  Fix a $G$-invariant positive definite continuous Hermitian form $\la \,,\,\ra_{\tilde \xi}$ on $W_{\tilde \xi}^\infty$.

\begin{lem}\label{ematrix}
The integrals in
\begin{equation}\label{hermh}
\begin{array}{rcl}
  W_{\tilde \xi}^\infty \times W_{\tilde \xi}^\infty &\rightarrow &\C,\\
    (u, v)&\mapsto & \int_H \la h.u, v\ra_{\tilde \xi}\,\sgn(h)^j\,\mathrm{d}h
  \end{array}
\end{equation}
converge absolutely, and yield a  continuous Hermitian form on $W_{\tilde \xi}^\infty$. Here ``$\mathrm{d}h$" denotes a Haar measure on $H$.

\end{lem}
\begin{proof}
Denote by $\Xi_K$ the Harish-Chandra $\Xi$-function (see \cite[Section 4.5.3]{Wa1}) on $G$ associated to the maximal compact subgroup $K$. By \cite[Theorem 1.2]{Sun2}, there is a continuous seminorm $\abs{\,\cdot\,}_{\tilde \xi}$ on $W_{\tilde \xi}^\infty$ such that
\[
\abs{ \la g.u,v\ra_{\tilde \xi}}\leq \Xi_K(g) \cdot \abs{u}_{\tilde \xi}\cdot \abs{v}_{\tilde \xi}, \quad \textrm{for all } u,v\in W_{\tilde \xi}^\infty, \, g\in G.
\]
By the estimate of $\Xi_K$ in \cite[Theorem 4.5.3]{Wa1}, it is easily verified that the restriction of $\Xi_K$ to $H$ is integrable. Therefore \eqref{hermh} is convergent and continuous. It defines a Hermitian form since $H$ is unimodular.
\end{proof}

 Fix an arbitrary element $g_C\in C\setminus C^\circ$. By the discussion of Section \ref{bl}, the space
 \begin{equation}\label{mino}
  \tau_{\tilde \xi}^+\oplus g_C.\tau_{\tilde \xi}^+
  \end{equation}
 forms the unique minimal $K$-type of $W_{\tilde \xi}^\infty$.
Write
\[
 G= SK
\]
for the Cartan decompostion of $G$, where
\[
S:=\{(x,y)\in G\mid \textrm{$x$ and $y$ are positive definite symmetric matrices}\}.
\]
As mentioned in the Introduction, the following result is a key ingredient of the proof.

\begin{lem}\label{positive} (see \cite[Theorem 1.5]{Sun1}) For every nonzero vector $u$ in the minimal $K$-type \eqref{mino} of $W_{\tilde \xi}^\infty$, the inequality
\[
  \la g.u, u\ra_{\tilde \xi}>0
\]
holds for all $g\in S$.
\end{lem}

By \eqref{lxi1} (for $\tilde \xi$) and
\eqref{taunn}, Lemma \ref{cartann} implies that  $(\tau_{\tilde \xi}^+)^{C^\circ}\neq 0$.
   Take a nonzero element
\[
  u_{\tilde \xi}\in (\tau_{\tilde \xi}^+)^{C^\circ}\subset \tau_{\tilde \xi}^+\subset W_{\tilde \xi}^\infty,
\]
and put
\[
  v_{\tilde \xi}:=\frac{u_{\tilde \xi}+ (-1)^j\, g_C. u_{\tilde \xi}}{2}.
\]
Then $v_{\tilde \xi}$ is a nonzero vector in the minimal $K$-type such that
\begin{equation}\label{gvxi}
  g.v_{\tilde \xi}=(\sgn(g))^j\, v_{\tilde \xi} ,\qquad \textrm{for all } g\in C.
\end{equation}

\begin{lem}\label{trc}
The Hermitian form \eqref{hermh} is positive definite on the one dimensional space $\C v_{\tilde \xi}$.
\end{lem}
\begin{proof}
By \eqref{gvxi} and Lemma \ref{positive}, we have that
  \begin{eqnarray*}
     &&\int_H  \la h. v_{\tilde \xi}, v_{\tilde \xi}\ra_{\tilde \xi} \,\sgn(h)^j\,\mathrm{d}h \\
     &=&  \int_{S\cap H} \int_C  \la sk. v_{\tilde \xi},v_{\tilde \xi}\ra_{\tilde \xi}\,\sgn(k)^j\,\mathrm{d}k\,\mathrm{d}s \\
     &=&   \int_{S\cap H}  \la s. v_{\tilde \xi},v_{\tilde \xi}\ra_{\tilde \xi}\,\mathrm{d}s >0. \\
       \end{eqnarray*}
Here ``$\mathrm{d}k$" denotes the normalized Haar measure on $C$, and ``$\mathrm{d}s$" is a certain positive measure on $S\cap H$. This proves the lemma.

\end{proof}
\begin{lem}\label{test2}
There exists an element of $\Hom_H(W_{\tilde \xi}^\infty, \sgn^{-j})$ which does not  vanish on
\[
 \tau_{\tilde \xi}^+\subset W_{\tilde \xi}^\infty.
\]
\end{lem}
\begin{proof}
The continuous linear functional
 \begin{equation}\label{zu1}
   W_{\tilde \xi}^\infty \rightarrow \C,\quad  u\mapsto \int_H  \la h. u,v_{\tilde \xi}\ra_{\tilde \xi}\,\sgn(h)^j\,\mathrm{d}h
 \end{equation}
belongs to $\Hom_H(W_{\tilde \xi}^\infty, \sgn^{-j})$. By Lemma \ref{trc}, it does not vanish on the minimal $K$-type $\tau_{\tilde \xi}^+\oplus g_C.\tau_{\tilde \xi}^+$. Then its $C$-equivariance further implies that it does not vanish on $\tau_{\tilde \xi}^+$.

\end{proof}

\subsection{The general case}

Note that $\tau_\xi^+$ is the PRV component of $\tau_{-\xi}^\vee\otimes \tau_{\tilde \xi}^+$. Using Casselman-Wallach's theory of smooth globalizations (see \cite{Cass}, \cite[Chapter
11]{Wa2} or \cite{BK}), Proposition \ref{transmu} and its analog for $\nu$ imply the following proposition.
\begin{prpl}\label{gxi}
The irreducible $K^\circ$-representation
 \[
   \tau_\xi^+\subset \tau_{-\xi}^\vee\otimes \tau_{\tilde \xi}^+
   \subset F_\xi\otimes W_{\tilde \xi}^\infty
 \]
 generates  an irreducible $G$-subrepresentation of  $F_\xi\otimes  W_{\tilde \xi}^\infty$ which is  isomorphic to $W_\xi^\infty$.
\end{prpl}

We are now prepared to prove Proposition \ref{test}. Take a nonzero element
\[
 \phi_F'\in \Hom_{H_\C}(F_\xi, {\det}^{-j}).
 \]
By Lemma \ref{nv10}, it does not vanish on $\tau_{-\xi}^\vee\subset F_\xi$. Using Lemma \ref{test2}, take a nonzero element
 \[
   \phi_{\tilde \xi}\in \Hom_H(W_{\tilde \xi}^\infty, \sgn^{-j})
 \]
 which does not vanish on $\tau_{\tilde \xi}^+\subset  W_{\tilde \xi}^\infty$. By Proposition \ref{gxi} and Proposition \ref{nprv}, the continuous linear functional
 \[
 \phi_F'\otimes   \phi_{\tilde \xi}: F_\xi\otimes W_{\tilde \xi}^\infty\rightarrow \abs{\det}^{-j}={\det}^{-j}\otimes \sgn^{-j}
 \]
 restricts to an element of $\Hom_H(W_\xi^\infty, \abs{\det}^{-j})$ which does not vanish on $\tau_\xi^+\subset W_\xi^\infty$. This finishes the proof of Proposition \ref{test},  in view of the following multiplicity one theorem.

\begin{lem}(\cite[Theorem B]{AzG} and \cite[Theorem B]{SZ})\label{mulonegl}
For every irreducible Casselman-Wallach representation $\pi_G$ of $G$ and every character $\chi_H$ of $H$, the inequality
\[
  \dim \Hom_H(\pi_G, \chi_H)\leq 1
\]
holds.
\end{lem}

\section{Proof of Theorem \ref{main}}\label{pmanin}

By \eqref{omegamu}, the set
\[
  \Omega(\xi):=\{\pi_n\widehat \otimes\pi_{n-1}\mid \pi_n\in \Omega(\mu),\,\pi_{n-1}\in \Omega(\nu)\}
\]
consists two irreducible representations of $G$. These two representations have the same restrictions to $G^\circ$, and to $H$.   Therefore, in order to prove Theorem \ref{main}, we may (and do) assume that $\pi_\xi=W_{\xi}^\infty$.

As in the discussion of Section \ref{nonvc}, we have that
\begin{equation}\label{coh3}
  \oH^{b_n+b_{n-1}}(\g, \tilde K^\circ; F_\xi^\vee\otimes W_{\xi}^\infty)=\Hom_{K^\circ}(\wedge^{b_n+b_{n-1}}(\g/\tilde \k), F_\xi^\vee\otimes W_{\xi}^\infty).
  \end{equation}
Likewise,
\begin{equation}
                \oH^{b_n+b_{n-1}}(\h, C^\circ; {\sgn}^j)=\Hom_{C^\circ}(\wedge^{b_n+b_{n-1}}(\h/\c), {\sgn}^j).
                     \end{equation}

Note that by Lemma \ref{prvc}, $\tau_{n,n-1}$ is the PRV component of $\tau_\xi^\vee \otimes \tau_\xi^+$. Write
\[
\varphi_\xi: \tau_{n,n-1}\rightarrow F_\xi^\vee\otimes W_{\xi}^\infty
\]
for the inclusion
\[
 \tau_{n,n-1} \subset \tau_\xi^\vee \otimes \tau_\xi^+\subset F_\xi^\vee\otimes W_{\xi}^\infty.
\]
Recall from the Introduction two nonzero elements
\[
  \phi_F\in \Hom_{H_\C}(F_\xi^\vee, {\det}^j)\quad\textrm{and}\quad \phi_\pi\in \Hom_H(W_\xi^\infty, \abs{\det}^{-j}).
\]

Recall the map $\eta_{n,n-1}: \wedge^{b_n+b_{n-1}}(\g/\tilde \k)\rightarrow \tau_{n,n-1}$ from \eqref{etan}. The composition of
\[
   \wedge^{b_n+b_{n-1}}(\g/\tilde \k)\xrightarrow{\eta_{n,n-1}} \tau_{n,n-1}\xrightarrow{\varphi_\xi} F_\xi^\vee\otimes W_{\xi}^\infty
\]
is an element of \eqref{coh3}. Its image under the map \eqref{homh} of Theorem \ref{main} equals the composition map
\begin{equation}\label{comhc0}
  \wedge^{b_n+b_{n-1}}(\h/\c) \xrightarrow{\iota_{n,n-1}} \wedge^{b_n+b_{n-1}}(\g/\tilde \k)\xrightarrow{\eta_{n,n-1}} \tau_{n,n-1}\xrightarrow{\varphi_\xi}  F_\xi^\vee\otimes W_{\xi}^\infty\xrightarrow{\phi_F\otimes \phi_\pi} \sgn^j.
\end{equation}
By Lemma \ref{nv1}, Proposition \ref{test} and Proposition \ref{nprv}, the composition of the last two arrows of \eqref{comhc0} is nonzero. Since it is $C^\circ$-invariant, it does not vanish on $\tau_{n,n-1}^{C^\circ}$. By Lemma \ref{nv2}, $\tau_{n,n-1}^{C^\circ}$ is equal to the image of the compositions of the first two arrows of \eqref{comhc0}. Therefore the composition \eqref{comhc0} is nonzero. This finishes the proof of Theorem \ref{main}.

\section{The case of complex groups}\label{complexg}

\subsection{The result}
Fix an integer $n\geq 2$ as before. Let $\K$ be a topological field which is isomorphic to $\C$, and write $\iota_1, \iota_2:\K\rightarrow \C$ for the two distinct isomorphisms.

 The notation of this section is different from that of previous ones. Fix a sequence
\[
  \mu=(\mu_1\geq \mu_2\geq \cdots \geq \mu_n;\, \mu_{n+1}\geq \mu_{n+2}\geq \cdots \geq \mu_{2n} )\in \Z^{2n}.
\]
Denote by $F_\mu$ the irreducible algebraic representation of $\GL_n(\C)\times \GL_n(\C)$ of highest weight $\mu$. It is also viewed as an irreducible representation of the real Lie group $\GL_n(\K)$ by restricting through the complexification map
\begin{equation}\label{compl}
  \GL_n(\K)\rightarrow \GL_n(\C)\times \GL_n(\C),\quad g\mapsto (\iota_1(g), \iota_2(g)).
\end{equation}

Denote by $\Omega(\mu)$ the set of isomorphism classes of irreducible Casselman-Wallach representations $\pi$ of
$\GL_{n}(\K)$ such that
\begin{itemize}
                           \item
                            $\pi|_{{\operatorname{SL}_n}(\K)}$ is unitarizable and
                            tempered, and
                           \item the total relative Lie algebra cohomology
                           \begin{equation}\label{cohome}
                             \oH^*(\gl_{n}(\C)\times \gl_n(\C),\GU(n);
                           F_\mu^\vee\otimes \pi)\neq 0,
                           \end{equation}
                             \end{itemize}
where  $\gl_{n}(\C)\times \gl_n(\C)$ is viewed as the complexification of $\gl_n(\K)$ through the differential of \eqref{compl}, and
\[
  \GU(n):=\{g\in \GL_n(\K)\mid \bar{g}^{\mathrm t} g\textrm{ is a scalar matrix}\}.
  \]

Similar to the real case, we have \cite[Section 3]{Clo},
\[
  \#(\Omega(\mu))=\left\{
                \begin{array}{ll}
                  0, & \hbox{if $\mu$ is not pure;} \\
                  1, & \hbox{if $\mu$ is pure.} \\
                                 \end{array}
              \right.
\]
Here ``$\mu$ is pure" means that
\begin{equation}
  \mu_{1}+\mu_{2n}=\mu_{2}+\mu_{2n-1}=\cdots=\mu_{n}+\mu_{n+1}.
\end{equation}
 Assume that $\mu$ is pure, and let $\pi_\mu$ be the unique representation in $\Omega(\mu)$.

Put
\[
  b_n:=\frac{n(n-1)}{2}.
\]
Then \cite[Lemma 3.14]{Clo}
\[
  \oH^b(\gl_{n}(\C)\times \gl_n(\C),\GU(n);
                           F_\mu^\vee\otimes \pi_\mu)=0,\qquad \textrm{if $b<b_n$,}
\]
and
\[
  \dim \oH^{b_n}(\gl_{n}(\C)\times \gl_n(\C),\GU(n);
                           F_\mu^\vee\otimes \pi_\mu)=1.
\]

We also fix a sequence
\[
  \nu=(\nu_1\geq \nu_2\geq \cdots \geq \nu_{n-1};\, \nu_n\geq \nu_{n+1}\geq \cdots \geq \nu_{2(n-1)} )\in \Z^{2(n-1)}.
\]
Define $F_\nu$ and $\Omega(\nu)$ similarly. Assume that $\nu$ is pure, and let $\pi_\nu$ be the unique representation in $\Omega(\nu)$.

Put
\[
G:=\GL_{n}(\K)\times \GL_{n-1}(\K),
\]
to be viewed as a real Lie group. View
\[
  H:=\GL_{n-1}(\K)
\]
as a subgroup of $G$ via the embedding
\begin{equation}\label{comp3}
g\mapsto \left(\left[
                      \begin{array}{cc}
                        g & 0 \\
                        0 & 1 \\
                      \end{array}
                    \right],   g\right).
\end{equation}
Similar to \eqref{compl},
\[
  G_\C:=(\GL_{n}(\C)\times \GL_{n}(\C))\times (\GL_{n-1}(\C)\times \GL_{n-1}(\C)),
\]
is a complexification of $G$, and
\[
  H_\C:=\GL_{n-1}(\C)\times \GL_{n-1}(\C)
\]
is a complexification of $H$. We view $H_\C$ as a subgroup of $G_\C$ by the complexification of the inclusion map \eqref{comp3}.

As in the real case, an element of  $\frac{1}{2}+\Z$ is called a critical place for $\pi_\mu\times \pi_\nu$ if it is not a pole of the local L-function $\oL(s, \pi_\mu \times \pi_\nu)$ or $\oL(1-s, \pi_\mu^\vee \times \pi_\nu^\vee)$.  Assume that $\mu$ and $\nu$ are compatible in the sense that there is an integer $j$ such that
 \begin{equation}\label{homf02}
  \Hom_{H_\C}(F_\xi^\vee, {\det}^j\otimes {\det}^j) \neq 0,
\end{equation}
where
\[
  F_\xi:=F_\mu \otimes F_\nu.
 \]
Note that the algebraic character ${\det}^j\otimes {\det}^j$ of $H_\C$ restricts to the character $\abs{\det}_\K^j$ of $H$, where ``$\abs{\,\,}_\K$" denotes the normalized absolute value of $\K$ (that is, $\abs{z}_\K=\iota_1(z)\iota_2(z)$ for all $z\in \K$).  Therefore, \eqref{homf02} is equivalent to
 \begin{equation}\label{homf2}
  \Hom_{H}(F_\xi^\vee, \abs{\det}_\K^j) \neq 0.
\end{equation}
Similar to the proof of \cite[Theorem 2.3]{KS}, one verifies that  $\frac{1}{2}+j$ is a critical place for $\pi_\mu\times \pi_\nu$, and conversely, all critical places are of this form under the assumption that $\mu$ and $\nu$ are compatible (see \cite[Theorem 2.21]{Rag2}).

 Fix a nonzero element $\phi_F$ of the hom space of \eqref{homf2}. As in the real case, the Rankin-Selberg integrals for $\pi_\mu\times \pi_\nu$ produce a nonzero element
\begin{equation}\label{homi}
  \phi_\pi\in \Hom_H(\pi_\xi, \abs{\det}_\K^{-j}),
\end{equation}
where $\pi_\xi:=\pi_\mu\widehat \otimes \pi_\nu$ is a Casselman-Wallach representation of $G$.

Write
\[
  \tilde K:=\GU(n)\times \GU(n-1)\subset G \quad \textrm{and}\quad C:=H\cap \tilde K=\oU(n-1)\subset H.
\]
Note that the cohomology spaces
\[
\oH^{b_n+b_{n-1}}(\g, \tilde K; F_\xi^\vee\otimes \pi_\xi)\quad\textrm{ and }\quad\oH^{b_n+b_{n-1}}(\h, C; \abs{\det}_\K^0)
\]
are both one dimensional. Here and as before, the complexified Lie algebra of a Lie group is denoted by the corresponding lower case gothic letter.

The nonvanishing hypothesis at the complex place is formulated as follows.
\begin{introtheorem}\label{main2}
By restriction of cohomology, the $H$-invariant linear functional
\[
 \phi_F\otimes \phi_\pi:  F_\xi^\vee\otimes \pi_\xi \rightarrow \abs{\det}_\K^0=\abs{\det}_\K^{j}\otimes \abs{\det}_\K^{-j}
\]
induces a linear map
\begin{equation}\label{homh2}
   \oH^{b_n+b_{n-1}}(\g, \tilde K; F_\xi^\vee\otimes \pi_\xi)\rightarrow
   \oH^{b_n+b_{n-1}}(\h, C; \abs{\det}_\K^0)
\end{equation}
which is nonzero.
\end{introtheorem}

\subsection{A sketch of proof}
The proof of Theorem \ref{main2} is similar to that of Theorem \ref{main}. We sketch a proof for the sake of completeness.

Recall that
\[
  \g^\K_n:=\g_n\times \g_n \qquad(\g_n:=\gl_n(\C))
\]
is viewed as the complexification of $\gl_n(\K)$. The corresponding complex conjugation  is given by
\begin{equation}\label{conk}
 \g^\K_n\rightarrow \g^\K_n,\quad  (x,y)\mapsto (\bar y, \bar x).
\end{equation}
The  Cartan involution corresponding to the maximal compact subgroup  $\oU(n)\subset \GL_n(\K)$ is
\begin{equation}\label{thetak}
  \g^\K_n\rightarrow \g^\K_n,\quad (x,y)\mapsto (-y^\mathrm t, -x^\mathrm t).
\end{equation}
Its fixed point set, which is the complexified Lie algebra of $\oU(n)$, is equal to
 \[
  \g_n^{\oU}:=\{(x,-x^{\mathrm t})\mid x\in \g_n\}.
\]

Recall the Borel subalgebra $\b_n=\t_n\ltimes \n_n$ of $\g_n$ from Section \ref{str}. Then the Borel subalgebra
\[
  \b^\K_n:=\b_n\times \b_n\subset \g_n^\K
\]
is ``theta stable" in the sense that it is stable under the Cartan involution \eqref{thetak}, and satisfies that
\[
  \b^\K_n\cap \bar \b^\K_n=\t^\K_n:=\t_n\times \t_n,
\]
where $\bar \b^\K_n:=\overline{\b_n}\times \overline{\b_n}$, which equals the image of $\b^\K_n$ under the map \eqref{conk}.

Similar to \eqref{dlambda}, for every $\lambda\in \R^{2n}\subset \C^{2n}=(\t^\K_n)^*$, put
\begin{equation}\label{dlambdac}
  \abs{\lambda}:=\textrm{the unique $\b_n^\K$-dominant element in the $\oW_{\g_n^\K}$-orbit of $\lambda$,}
\end{equation}
where $\oW_{\g_n^\K}$ denotes the Weyl group of $\g_n^\K$ with respect to the Cartan subalgebra $\t_n^\K$.
Denote by $2\rho_n\in(\t^\K_n)^*$ the sum of all weights of $\n_n^\K:=\n_n\times \n_n$.

As in \eqref{vmu},
\[
  V_\mu:=\oU(\g_n^\K)\otimes_{\oU(\bar \b_n^\K)} \C_{\abs{\mu}+2\rho_n}
  \]
  is an irreducible $(\g_n^\K, \oT_n^{\oU})$-module, where $\oT_n^{\oU}$ is the Cartan subgroup of $\oU(n)$ with complexified Lie algebra
  \[
  \t_n^{\oU}:=\t_n^\K\cap \g_n^{\oU}.
\]

Denote by $\Pi$ the Bernstein functor from the category of $(\g_n^\K, \oT_n^{\oU})$-modules to the category of $(\g_n^\K, \oU(n))$-modules, and write $\Pi_i$ for its $i$-th left derived functor ($i\in \Z$). Then
\[
  \Pi_i(V_\mu)=0\quad\textrm{unless}\quad i=S_n:=\dim \n_n=\frac{n(n-1)}{2},
  \]
and the Casselman-Wallach smooth globalization $W_\mu^\infty$ of
\[
 W_\mu:=\Pi_{S_n}(V_\mu)
\]
is isomorphic to $\pi_\mu$.

For every $\lambda\in (\t^\K_n)^*$, write $[\lambda]\in (\t_n^{\oU})^*$ for its restriction to $\t_n^{\oU}$.  Denote by $\tau_\mu$ and $\tau_{-\mu}$ the irreducible representations of $\oU(n)$ of highest weights $[\abs{\mu}]$ and $[\abs{-\mu}]$, respectively. Write $2\rho_n^{\oU}\in (\t_n^{\oU})^*$ for the sum of all weights of
\[
  \n_n^{\oU}:=\n_n^\K\cap \g_n^{\oU}.
\]
Denote by $\tau_n$  the irreducible representation of $\oU(n)$ of highest weight
\[
[2\rho_n]-2\rho_n^{\oU}=2\rho_n^{\oU}.
\]
Write $\tau_\mu^+$ for the Cartan product of $\tau_{\mu}$ and $\tau_n$. Then $\tau_\mu^\vee$, $\tau_{-\mu}^\vee$  and $\tau_\mu^+$ occur with multiplicity one in $F_\mu^\vee$,  $F_\mu$ and $W_\mu^\infty$, respectively.

Put
\[
 \tilde \mu:=(\mu_1-\mu_n, \mu_2-\mu_{n-1},\cdots, \mu_n-\mu_1; \mu_{n+1}-\mu_{2n}, \mu_{n+2}-\mu_{2n-1},\cdots, \mu_{2n}-\mu_{n+1})
\]
so that $F_{\tilde \mu}$ is the Cartan product of $F_\mu$ and $F_\mu^\vee$.
Similar to Proposition \ref{transmu}, the irreducible $\oU(n)$-representation
 \[
   \tau_\mu^+\subset \tau_{-\mu}^\vee\otimes \tau_{\tilde \mu}^+
   \subset F_\mu\otimes W_{\tilde \mu}
 \]
 generates  an irreducible $(\g_n^\K,\oU(n))$-module  which is  isomorphic to $W_\mu$. Applying the above argument to $\nu$, we get
 spaces
 \[
   \tau_\nu^+\subset \tau_{-\nu}^\vee\otimes \tau_{\tilde \nu}^+
   \subset F_\nu\otimes W_{\tilde \nu},
 \]
 and $\tau_\nu^+$ generates  an irreducible $(\g_{n-1}^\K,\oU(n-1))$-submodule of  $F_\nu\otimes W_{\tilde \nu}$ which is  isomorphic to $W_\nu$.

 Note that the analog of Proposition \ref{nprv} also holds for unitary groups. Therefore by using translations, the same proof as in Section \ref{ctest} shows that the functional $\phi_\pi$ does not vanish on
 \[
   \tau_{\xi}^+:=\tau_\mu^+\otimes \tau_{\nu}^+\subset W_\xi^\infty:=W_\mu^\infty\widehat \otimes W_\nu^\infty.
 \]
Similar to Lemma \ref{nv1}, the functional $\phi_F$ does not vanish on
\[
  \tau_{\xi}^\vee:=\tau_{\mu}^\vee\otimes \tau_{\nu}^\vee\subset F_\xi^\vee=F_\mu^\vee \otimes F_\nu^\vee.
\]

As in Lemma \ref{ngln2}, the representation $\tau_{n,n-1}:=\tau_n\otimes \tau_{n-1}$ of $\oU(n)\times \oU(n-1)$ occurs with multiplicity one in $\wedge^{b_n+b_{n-1}}(\g/\tilde \k)$. Similar to \eqref{comhc0}, we have a sequence
\begin{equation}\label{comf}
  \wedge^{b_n+b_{n-1}}(\h/\c) \rightarrow \wedge^{b_n+b_{n-1}}(\g/\tilde \k)\rightarrow \tau_{n,n-1}\rightarrow  F_\xi^\vee\otimes W_{\xi}^\infty\xrightarrow{\phi_F\otimes \phi_\pi} \abs{\det}_\K^0.
\end{equation}
Finally, as in Section  \ref{pmanin}, the composition of \eqref{comf} is nonzero. This finishes the proof of Theorem \ref{main2}.

\section*{Acknowledgements}
The author is grateful to  Michael Harris for the constant encouragements and the suggestion of using translation functors to attack the problem. He thanks Fabian Januszewski, Joachim Mahnkopf and A. Raghuram for comments, and thanks Hendrik Kasten for many English corrections to an early  manuscript of this paper.
Finally the author would like to thank the referees for providing very insightful comments to improve the paper. The work was supported by NSFC Grants 11525105, 11321101 and 11531008.


\begin{thebibliography}{99}
\bibitem[AzG]{AzG}
A. Aizenbud and D. Gourevitch, \textit{Multiplicity one theorem for $({\GL}_{n+1}(\R),{\GL}_n(\R))$}, Selecta Math. (N.S.) 15, (2009), 271-294.


\bibitem[BN]{BN}
A. Borel and N. Wallach, \textit{Continuous cohomology, discrete subgroups, and representations
of reductive groups}, second edition, Mathematical Surveys and Monographs, 67. American
Mathematical Society, Providence, RI, 2000.


\bibitem[BK]{BK}
J. Bernstein and B. Kr\"{o}tz, \textit{Smooth Frechet globalizations of
Harish-Chandra modules},  Isr. J. Math. 199 (2014), 45-111.


\bibitem[Cas]{Cass}
W. Casselman, \textit{Canonical extensions of Harish-Chandra moudles
to representations of $G$}, Can. Jour. Math. 41, (1989), 385-438.

\bibitem[Clo]{Clo}
L. Clozel, \textit{Motifs et formes automorphes: applications du
principe de fonctorialit\'{e}}.(French) [Motives and automorphic
forms: applications of the functoriality principle] Automorphic
forms, Shimura varieties, and L-functions, Vol. I (Ann Arbor, MI,
1988), 77-159, Perspect. Math., 10, Academic Press, Boston, MA,
1990.


\bibitem[Di]{Di}
J. Dixmier, \textit{Enveloping algebras}, Graduate Studies in Mathematics 11, American Mathematical
Society, Providence, RI, 1996.

\bibitem[Fl]{Fl}
M. Flensted-Jensen, \textrm{Discrete series for semisimple symmetric spaces}, Ann. of Math. (2) 111 (1980),
253-311.


\bibitem[GH]{GH}
H. Grobner and M. Harris, \textit{Whittaker periods, motivic periods, and special values of tensor product L-functions}, J. Inst. Math. Jussieu (2015), 1-59.


\bibitem[He1]{He1} E. Hecke, \textit{\"{U}ber die Bestimmung Dirichletscher Reihen durch ihre Funktionalgleichung},
Mathematische Annalen 112 (1936), 664-699.

\bibitem[He2]{He2} E. Hecke, \textit{Modulfunktionen und Dirichletschen Reihen mit Eulerscher Produktentwicklung
I.}, Mathematische Annalen 114 (1937), 1-28.

\bibitem[He3]{He3}
E. Hecke, \textit{\"{U}ber Modulfunktionen und die Dirichletschen Reihen mit Eulerscher Produktentwicklung
II.}, Mathematische Annalen 114 (1937), 316-351.


\bibitem[Jac]{Jac}
H. Jacquet, \textit{Archimedean Rankin-Selberg integrals}, in
\textit{Automorphic Forms and $L$-functions II: Local Aspects},
Proceedings of a workshop in honor of Steve Gelbart on the occasion
of his 60th birthday, Contemporary Mathematics, volumes 489,
57--172, AMS and BIU 2009.

\bibitem[Jan1]{Jan1}
F. Januszewski, \textit{Modular symbols for reduc-
tive groups and p-adic Rankin-Selberg convolutions over number fields},
J. reine angew. Math. 653 (2011), 1-45.

\bibitem[Jan2]{Jan2}
F. Januszewski, \textit{On p-adic L-functions for $\GL(n)\times  \GL(n-1)$ over totally real fields}, Int Math Res Notices 17 (2015), 7884-7949.

\bibitem[Jan3]{Jan3}
F. Januszewski, \textit{p-adic L-functions for Rankin-
Selberg convolutions over number fields}, arXiv:1501.04444.

\bibitem[Jan4]{Jan4}
F. Januszewski, \textit{On period
relations for automorphic L-functions}, arXiv:1504.06973.

\bibitem[KS]{KS}
H. Kasten and C.-G. Schmidt, \textit{On critical values of
Rankin-Selberg convolutions}, International Journal of Number Theory, 9 (2013), 205-256.

\bibitem[KMS]{KMS}
D. Kazhdan, B. Mazur and C.-G. Schmidt, \textit{Relative modular
symbols and Rankin-Selberg convolutions}, J. reine angew. Math. 519
(2000), 97-141.

\bibitem[K]{K}
A. Knapp, \textit{Representation Theory of Semisimple Groups: An Overview Based on Examples}, Princeton University Press, Princeton, N.J., 1986.

\bibitem[KV]{KV}
A. Knapp and D. Vogan, \textit{Cohomological induction and unitary
Representations}, Princeton University Press, Princeton, N.J., 1995.



\bibitem[Mah]{Mah}
J. Mahnkopf, \textit{Cohomology of arithmetic groups, parabolic
subgroups and the special values of L-functions of $\GL_n$}, J.
Inst. Math. Jussieu., 4(4), 553-637 (2005).


\bibitem[PRV]{PRV}
K.R. Parthasarathy, R. R. Rao and V.S. Varadarajan, \textit{Representations
of complex semi-simple Lie groups and Lie algebras}, Ann. of Math. 85 (1967),
383-429.


\bibitem[Rag1]{Rag}
A. Raghuram, \textit{On the special values of certain Rankin-Selberg
L-functions and applications to odd symmetric power L-functions of
modular forms}, Int. Math. Res. Not. Vol. (2010) 334-372,
doi:10.1093/imrn/rnp127.

\bibitem[Rag2]{Rag2}
A. Raghuram, \textrm{Critical values of Rankin-Selberg L-functions for $\GL(n) \times \GL(n-1)$ and the symmetric cube L-functions for $\GL(2)$}, Forum Math., DOI: 10.1515/forum-2014-0043.

\bibitem[RS1]{RS1}
A. Raghuram and F. Shahidi, \textit{Functoriality and special values of
L-functions}, in: \textit{Eisenstein series and Applications}, eds
W. T. Gan, S. Kudla, Y. Tschinkel, Progress in Mathematics 258
(Boston, 2008), pp. 271-294.

\bibitem[RS2]{RS}
A. Raghuram and F. Shahidi, \textit{On certain period relations for
cusp forms on $\GL_n$}, Int. Math. Res. Not. Vol. (2008),
doi:10.1093/imrn/rnn077.



\bibitem[Sch1]{Sch}
C.-G. Schmidt, \textit{Relative modular symbols and p-adic
Rankin-Selberg convolutions}, Invent. Math. 112 (1993), 31-76.


\bibitem[Sch2]{Sch2}
C.-G. Schmidt, \textit{Period relations
and p-adic measures}, manuscr. math. 106 (2001), 177-201.


\bibitem[Sun1]{Sun1}
B. Sun, \textit{Positivity of matrix coefficients of representations with real infinitesimal characters}, Isr. J. Math. 170 (2009), 395-410.

\bibitem[Sun2]{Sun2}
B. Sun, \textit{Bounding matrix coefficients for smooth vectors of tempered representations},
Proc. Amer. Math. Soc. 137 (1), (2009) 353-357.

\bibitem[SZ]{SZ}
B. Sun and C.-B. Zhu, \textit{Multiplicity one theorems: the Archimedean case}, Ann. of Math. (2) 175 (2012), 23-44.

\bibitem[V1]{V}
D. Vogan, \textit{Unitarizability of certain series of representations}, Ann. of
Math. 120 (1984), 141-187.

\bibitem[V2]{V2}
D. Vogan, \textit{The unitary dual of GL(n) over an archimedean field}, Invent. Math. 83 (1986), 449-505.

\bibitem[VZ]{VZ}
D. Vogan and G. Zuckerman, \textit{Unitary representations with
non-zero cohomology}, Compositio Math. 53 (1984)), 51-90.

\bibitem[Wa1]{Wa1}
N. Wallach, \textit{Real Reductive Groups I}, Academic Press, San
Diego, 1988.

\bibitem[Wa2]{Wa2}
N. Wallach, \textit{Real Reductive Groups II}, Academic Press, San
Diego, 1992.

\bibitem[Wa3]{Wa3}
N. Wallach, \textit{ On the unitarizability of derived functor modules, Invent. Math.} 78 (1984), 131-141.

\bibitem[Ya]{Ya}
O. Yacobi, \textit{An analysis of the multiplicity spaces in branching of symplectic groups}, Selecta Math N.S., 16 (2010), 819-855.

\end{thebibliography}
\end{document}